\documentclass[12pt]{article}

\usepackage{amsfonts,latexsym,amssymb,amsmath,amsthm,authblk,color,hyperref}
\usepackage[pagewise]{lineno}

\date{\today}
\def \al{\alpha}

\def \ga{\gamma}
\def \dl{\delta}
\def \ep{\varepsilon}

\def \la{\lambda}
\def \si{\sigma}

\def \om{\omega}

\def \Si{\Sigma}

\def \Om{\Omega}

\def \operatorname#1{\mathop{\rm #1}}

\def\div{\operatorname{div}}
\def\osc{\operatorname{osc}}
\def\osc2{\operatorname{osc^2}}
\def\dist{\operatorname{dist}}
\def\supp{\operatorname{supp}}

\def\cd{\partial}
\def\esssup{\operatorname{esssup}}

\def\Q0{Q(x_0,t_0,R)}
\def\0{{x_0,t_0,R}}
\def\build#1_#2{\mathrel{\mathop{\kern 0pt#1}\limits_{#2}}}

\newtheorem{theorem}{Theorem}[section]

\newtheorem{proposition}{Proposition}[section]
\newtheorem{lemma}{Lemma}[section]
\newtheorem{definition}{Definition}[section]

\begin{document}

\title{Elliptic   equations with a singular drift  \\ from a weak Morrey space}
\author{Misha Chernobai\footnote{Department of Mathematics, UBC, Canada, mchernobay@gmail.com}, Tim  Shilkin\footnote{V.A.~Steklov Mathematical Institute, St.-Petersburg, Russia,  and Max Planck Institute for Mathematics in the Sciences, Leipzig, Germany,  tim.shilkin@gmail.com}}
\maketitle

 \abstract{In this paper we prove the existence, uniqueness and regularity of weak solutions to the Dirichlet problem for an elliptic equation with a drift $b$ satisfying $\div b\le 0$ in $\Om$. We assume $b$ belongs to some weak Morrey class which  includes in the 3D case, in particular,   drifts having a singularity along the axis $x_3$  with the asymptotic $b(x)\sim c/r$, where $r=\sqrt{x_1^2+x_2^2}$.}

\section{Introduction and Main Results}

\bigskip  Assume  $\Om \subset  \Bbb R^n$ is a bounded Lipschitz domain, $n\ge 3$.     We consider the
following boundary value problem:
\begin{equation}
\left\{ \quad \gathered     -\Delta u +  b \cdot \nabla u \, = \, -\div f \qquad\mbox{in}\quad \Om, \\
u|_{\cd \Om} \ = \ 0. \qquad \quad
\endgathered\right.
\label{Equation}
\end{equation}
Here   $u:  \Om\to \Bbb R$ is unknown,   $b:\Om\to \Bbb R^n$ and
$f: \Om \to \Bbb R^n$ are given functions.
In particular, we are interested in the drift of type
\begin{equation} n=3, \qquad b(x) =  -\al  \, \frac {x'}{|x'|^2}, \qquad x'=(x_1, x_2, 0), \qquad \al\in \Bbb R.
\label{Typical_b}
\end{equation}
Motivated by this example, we assume that the  drift $b$ satisfies the condition
\begin{equation}
b\in L^{2, n-2}_w(\Om),  \label{Assumptions_Morrey_space}
\end{equation}
where
$L^{p,\la}_w(\Om)$ is the weak Morrey  space  equipped
with the quasi-norm
$$
\| b\|_{L^{p,\la}_w(\Om)} \ := \ \sup\limits_{x_0\in
\Om}\sup\limits_{R<\operatorname{diam}\Om} \, R^{-\frac \la p} \,
\|b\|_{L_{p,w}(B_R(x_0)\cap \Om)},
$$
and $L_{p,w}(\Om)$ is the weak Lebesgue space equipped with the quasi-norm
\begin{equation}
\| b\|_{L_{p,w}(\Om)} \ := \ \sup\limits_{s>0}
s~|\{~x \in \Om: ~|b(x )|>s~\}|^{\frac 1p},
\label{Weak_Lebesgue}
\end{equation}
where $|\Om|$ denotes the $n$-dimensional Lebesgue measure of $\Om$.
 Note that the condition \eqref{Assumptions_Morrey_space} implies in particular that $b$ belongs to the usual Morrey space $L^{p, n-p}(\Om)$ for any $p\in [1, 2)$ (see  Proposition \ref{Holder_inequality} below and the definition of Morrey spaces at the end of this section).
  Note also that  both scales  of  spaces  $L^{p,n-p}(\Om)$ and  $L^{p,n- p}_w (\Om)$ are  critical ones,  i.e. these spaces  are  invariant under the   scaling
\begin{equation}\label{Scaling}
b^{ \la }(x) = \la  \, b( \la  x),  \quad  \la >0 \quad \Longrightarrow \quad  \quad \| b^{  \la } \|_{L^{p,n-p}_w ( \Om_\la  ) } =\| b  \|_{L^{p,n-p}_w (\Om )},
\end{equation}
 where we assume $\Om$ is star-shaped and   denote $\Om_{ \la}:=\{ \, x\in \Bbb R^n:\,  x/ \la  \in \Om\, \}$.
Concerning further properties of   weak Morrey spaces we refer to \cite{Di_Fazio_2020}, \cite{SFH}. Here we emphasise only that the scale of weak Morrey spaces $L_w^{p, n-p}(\Om)$  is convenient for the description  of drifts $b$   which have a certain asymptotic near   some singular submanifold $\Si\subset \Om$. For example, for $n=3$ we have
$$
\begin{array}{ccl}
b(x) \sim \frac1{x_3},& \Si \ = \mbox{ a plane } & \quad \Longrightarrow \quad  b\in L^{1,2}_w(\Om),  \phantom{\Big|}\\
b(x) \sim \frac1{\sqrt{x_1^2+x_2^2}},& \Si \  = \mbox{ a line } &  \quad \Longrightarrow \quad b\in L^{2,1}_w(\Om), \phantom{\Big|}\\
b(x) \sim \frac1{\sqrt{x_1^2+x_2^2+x_3^2}},& \Si \ = \mbox{ a point } & \quad \Longrightarrow \quad  b\in L^{3,0}_w(\Om)=L_{3,w}(\Om). \phantom{\Big|}\\
\end{array}
$$

We define the bilinear form $\mathcal B[u,\eta]$ by
\begin{equation}
\mathcal B[u,\eta] \ := \ \int\limits_{\Om}  \eta \,  b   \cdot \nabla u  ~dx. \label{Bilinear_Form}
\end{equation}
Note that for
$b $   satisfying \eqref{Assumptions_Morrey_space} the bilinear form   \eqref{Bilinear_Form}
generally speaking is not well-defined  for $u\in W^{1 }_2(\Om)$ and $\eta\in C_0^\infty(\Om)$ where
$$
W^1_p(\Om) \, := \, \{\, u\in L_p(\Om): \, \nabla u \in L_p(\Om)\, \}, \qquad  p\in [1, +\infty],
$$
is the usual Sobolev space and $C_0^\infty(\Om)$ is the linear space of $C^\infty$--smooth functions compactly supported in $\Om$.

Nevertheless, the bilinear form $\mathcal B[u,\eta]$ is
well-defined at least   for $u\in W^{1 }_p(\Om)$ with $p>2$ and
$\eta\in L_q(\Om)$ with $q>\frac {2p}{p-2}$.  So,
instead of the standard notion of weak solutions from the energy
class $W^{1 }_2(\Om)$   we introduce the definition of $p$-weak
solutions to the problem \eqref{Equation}, see also \cite{Kang_Kim}, \cite{Tsai},  \cite{Kwon_1}:

\begin{definition}
Assume  $p> 2$, $b\in L_{p'}(\Om)$, $p'=\frac{p}{p-1}$, and  $f\in L_1(\Om)$.
 We say $u$ is a $p$-weak solution to the problem
\eqref{Equation} if $ u\in   \overset{\circ}{W}{^1_p}(\Om) $
and $u$ satisfies the identity
\begin{equation}
\gathered  \int\limits_{\Om}   \nabla u \cdot
\nabla\eta \, dx  \ + \ \mathcal B[u,\eta] \ = \
\int\limits_{\Om} f\cdot \nabla \eta~dx   , \qquad  \forall~\eta\in C_0^\infty(\Om).
\endgathered
\label{Identity}
\end{equation}
If  $b\in L_2(\Om)$  and  $u\in   \overset{\circ}{W}{^1_2}(\Om)$ satisfy  \eqref{Identity} then we call $u$ a weak solution to the problem  \eqref{Equation}. Obviously,  in this case  $p$-weak solutions are some subclass of weak solutions.

\end{definition}
Note that if $u$ is a weak or a $p$-weak solution to \eqref{Equation} and $f\in L_2(\Omega)$ then by density arguments we can extend the class of test functions in \eqref{Identity} from $\eta\in C_0^\infty(\Om)$ to all functions $\eta\in\overset{\circ}{W}{^1_2}(\Om)\cap L_{\infty}(\Omega)$.

In this paper we   always assume that the vector field $b$  has a sign-defined divergence:
\begin{equation}
 \div b  \, \le  \, 0 \quad \mbox{in} \quad
\mathcal D'(\Om), \label{Assumptions_b}
\end{equation}
where $\mathcal D'(\Om)$ is the space of distributions on $\Om$.
In the case \eqref{Typical_b} the condition \eqref{Assumptions_b} corresponds to $\al\ge 0$.
The important impact of the condition  \eqref{Assumptions_b} is due to the fact that in this case  the quadratic form $\mathcal B[u,u]$ provides   a positive support to the quadratic form of the elliptic operator in \eqref{Equation} (see Proposition \ref{Q_Form_good_sign} below), while in the opposite case $\div b\ge 0$ in $\Om$   the quadratic form $\mathcal B[u,u]$  is (formally) non-positive and hence it ``shifts'' the operator to the ``spectral area''. It is well-known that in the case of a non-regular drift $b$ the violation of the condition \eqref{Assumptions_b} potentially can  ruin   the uniqueness for  the problem \eqref{Equation}   even in the class of   smooth solutions. For example, in the case $\Om=B$ where $B:=\{\, x\in \Bbb R^n:\, |x|<1\,\}$ the functions
\begin{equation}\label{Counterexample}
u(x) = c(|x|^2 -1) ,   \quad b(x) =n \, \tfrac{x}{|x|^2}, \quad  b\in L_{n,w}(B), \quad \div b \ge 0  \ \   \mbox{ in } \ \ \mathcal D'(B)
\end{equation}
satisfy \eqref{Equation} with $f\equiv 0$.  Note also that in the regular case   $b\in L_n(\Om)$ the uniqueness for the problem \eqref{Equation}  holds  regardless of the sign of  $\div b$ due to the maximum principle, see, for example, \cite{Filonov_Shilkin} and reference there.

\medskip
The main results of the present paper are  the following two
theorems:

\begin{theorem}\label{Theorem_1} Assume   $n\ge 3$,  $\Om\subset \Bbb R^n$ is a bounded Lipschitz domain and $b$   satisfies
\eqref{Assumptions_Morrey_space}, \eqref{Assumptions_b}. Then there exist  $p>2$
depending only on  $n$, $\Om$  and  $\| b\|_{L^{2, n-2}_{w}(\Om)}$  such
that  for any $f\in L_p(\Om)$  there exists a unique $p$-weak
solution $u\in W^1_p(\Om)$ to
  the problem \eqref{Equation}. Moreover, this solution satisfies the estimate
\begin{equation}
\| u\|_{W^1_p(\Om)} \ \le \ c~\| f\|_{L_p(\Om)},
\label{Main_Estimate}
\end{equation}
 with a
constant $c>0$ depending only on     $n$, $\|
b\|_{L^{2, n-2}_{w}(\Om)} $ and the Lipschitz constant of $\cd \Om$.
\end{theorem}

\begin{theorem}\label{Theorem_3}   Assume   $n\ge 3$,  $\Om\subset \Bbb R^n$ is a bounded Lipschitz domain and  $b$   satisfies
 \eqref{Assumptions_Morrey_space},  \eqref{Assumptions_b} and  $p>2$.  Then for any
  $q>n$ there    exists    $\mu\in (0,1)$  depending only on  $n$, $p$, $q$, $\Om$  and  $\| b\|_{L^{2, n-2}_w(\Om)}$
such that
if $u$ is a $p$--weak solution to the problem \eqref{Equation} corresponding to the right-hand side $f\in L_q(\Om)$  then    $u$ is H\" older continuous on $\bar \Om$ with the exponent $\mu$  and the  estimate
\begin{equation}
  \|u\|_{C^\mu(\bar \Om)} \ \le \ c~\| f\|_{L_q(\Om)},
\label{Holder_Estimate}
\end{equation}
 holds  with the
constant $c>0$ depending only on  $q$, $p$, $n$,   $\|
b\|_{L^{2, n-2}_{w}(\Om)} $ and the Lipschitz constant of $\cd \Om$.
\end{theorem}

The study of the weak solvability and properties of weak solutions to drift-diffusion equations has a long history. Below we focus on the brief overview of   results concerning the existence of weak solutions, their uniqueness,   boundedness and H\" older continuity.

Very often the study of the problem \eqref{Equation} is coupled with the study of the dual problem
  \begin{equation}
\left\{ \quad \gathered     -\Delta v -  \div (b v)  \, = \,   \div g \qquad\mbox{in}\quad \Om, \\
v|_{\cd \Om} \ = \ 0. \qquad \quad
\endgathered\right.
\label{Equation_dual}
\end{equation}
In the case of the divergence--free drift
\begin{equation} \label{Div-free_drift}
\div b=0 \quad \mbox{in} \quad \mathcal D'(\Om)
\end{equation}
problems \eqref{Equation} and \eqref{Equation_dual} are equivalent, but for the general drifts it is not the case. The reason is that    the dual problem \eqref{Equation_dual} is presented in the divergent form. Hence   weak solutions, depending on the regularity of $b$,  in principle are well-defined in a wider class than the energy space $W^1_2(\Om)$ (one can consider classes of  very weak solutions, renormalized solutions etc.)   For example, in a series of papers \cite{B_2009}, \cite{B_2015}, \cite{BO} (see also \cite{BDGC} and reference there) the existence of weak (distributional) solutions  for the problem \eqref{Equation_dual}  in various functional classes under various critical   and supercritical assumptions on $b$ was studied.

Another approach to the solvability of the problem \eqref{Equation} with the rough coefficients was developed in a series of papers \cite{Krylov_0}, \cite{Krylov_1}, \cite{Krylov_2}, \cite{Krylov_3} (see also  references there). This approach is applicable to   elliptic operators in the non-divergent form  (with  the main part $-a_{jk}(x)\frac{\cd^2 u}{\cd x_j\cd x_k}$ with a non-smooth  uniformly elliptic matrix $a(x)=(a_{jk}(x))$  instead of the Laplacian) and the drift $b$ belonging to the critical Morrey spaces. In this approach the existence of solutions in   Sobolev spaces $W^2_p(\Om)$ with some $p>1$ is investigated and the equation \eqref{Equation} is understood  a.e. in $\Om$.

For the problem \eqref{Equation} weak solutions from $W^1_2(\Om)$ formally are well-defined only for   $b\in L_2(\Om)$ which does not cover the case  \eqref{Assumptions_Morrey_space}.  Nevertheless, weak solutions to elliptic and parabolic equations with rough coefficients are commonly known to possess the {\it higher integrability} property. This means that the gradient of a weak solution is integrable with some exponent which is slightly greater than 2. One can ask  what are the optimal conditions on the drift $b$ which provide  the higher integrability of weak solutions to the problem \eqref{Equation}? The first results in this direction were obtained earlier in \cite{Kim_Kim}, \cite{Tsai}, \cite{Kwon_1} where the critical drift $b\in L_{n,w}(\Om)$ satisfying \eqref{Assumptions_b} was studied. So, our Theorem \ref{Theorem_1} can be viewed as a far away extension of these results. Note also that the 2D case was investigated separately in \cite{Chernobai_Shilkin} and \cite{Kwon_2}.

All results concerning weak solvability and properties of weak solutions to the problems \eqref{Equation} and \eqref{Equation_dual} formally can be split onto three groups: the results related to the subcrtitical drifts, the results for critical drifts and supercritical results. We remind that we call some functional class $X$ supercritical if for $b\in X$ and $b^\la$ defined in \eqref{Scaling} the norm $\| b^\la \|_{X(B)}$ computed over a unit ball $B$ blows up as $\la\to 0$.

The results concerning the drift $b$ belonging to subcritical Lesbegue spaces are classical one, see \cite{LU}. The same is true for   drifts belonging to the critical Lebesgue space
\begin{equation}\label{Regular_drift}
b\in L_n(\Om).
\end{equation}
We    call drifts   satisfying \eqref{Regular_drift}   {\it regular}. In the regular case  the drift term provides a compact perturbation of the Laplace operator and hence the problem \eqref{Equation} possesses the Fredholm property (i.e. the existence for \eqref{Equation} follows from the uniqueness). Moreover, in the   case \eqref{Regular_drift}  weak solutions  are H\" older continuous  and  the maximum principle holds (which provides the uniqueness of weak solutions), see the survey of the results, for example, in \cite{Filonov_Shilkin}. So, we can consider the regular case as a starting point of our investigation. Note   that the necessary and sufficient conditions for the boundedness of the bilinear form \eqref{Bilinear_Form} on the Sololev space $W^1_2(\Bbb R^n)$ were investigated in \cite{MV}.

The counterexample \eqref{Counterexample} shows that even the minimal relaxation of the condition \eqref{Regular_drift} (still within the critical scale) can ruin the maximum principle and   the uniqueness for the problem \eqref{Equation}. So, to recover the uniqueness we need to impose some additional assumptions  on $b$. For example, one can consider the condition \eqref{Assumptions_b} or its particular case \eqref{Div-free_drift}. (Motivated by the properties of the quadratic form $\mathcal B[u,u]$  we  call the  condition \eqref{Assumptions_b}  ``non-spectral''.) In the present paper we focus only on the critical ``non-spectral'' case and we are going to investigate the critical ``spectral'' case $\div b\ge 0$ in a separate forthcoming paper, see  \cite{Chernobai_Shilkin} and \cite{Kwon_2}  for the related 2D results.

For $b\in L_2(\Om)$ satisfying \eqref{Div-free_drift} the uniqueness of weak solutions in the energy class $ W ^1_2 (\Om)$ can be found in \cite{Zhikov} (see also \cite{Zhang} for the parabolic case). In \cite{Zhikov} a counterexample to the uniqueness for the problem \eqref{Equation} with a divergence-free drift was also constructed if the condition $b\in L_2(\Om)$ is violated (see \cite{Tomek} and references there for the further developments in this direction).
For $b\in L_2(\Om)$ satisfying  the ``non-spectral'' condition \eqref{Assumptions_b} the corresponding uniqueness result in the energy class $W^1_2(\Om)$ can be easily obtained by a straightforward adaptation of the methods of \cite{Zhang}, \cite{Zhikov}. For $b\in L_{2,w}(\Om)$ satisfying \eqref{Assumptions_b} our Proposition \ref{Theorem_2} below provides a   uniqueness result for the problem \eqref{Equation} in a class of weak solutions which possess the higher integrability property.

The next issue is the boundedness of weak solutions. It is well-known that even in the divergence--free case \eqref{Div-free_drift} the {\it local} boundedness of weak solutions requires some extra regularity of the drift, see \cite{NU}, \cite{Filonov_Shilkin}, \cite{Filonov_Hodunov}, \cite{AD}, see also a recent survey \cite{AN}  and reference there. Nevertheless, our Proposition \ref{Theorem_2} shows that for a drift $b\in L_{2,w}(\Om)$ satisfying \eqref{Assumptions_b}  $p$--weak solutions to the boundary value problem \eqref{Equation} are always bounded without any extra assumptions on the regularity of $b$. We emphasize that our result is global (i.e. it requires   smooth boundary conditions) and can not be localized.  The phenomena of the dependence of local boundedness  of weak solutions on their behaviour on the boundary of the domain is discussed in \cite{Filonov_Shilkin}.

The last issue we are going to discuss is H\" older continuity of weak solutions. Note that under the critical condition \eqref{Regular_drift} H\" older continuity is the  optimal regularity.
For the problems \eqref{Equation}, \eqref{Equation_dual} with a drift $b$ belonging to the subcritical Morrey space $b\in L^{2, \la}(\Om)$ with
 $\la\in (n-2,n)$   H\" older continuity of weak solutions is known for a long time, see \cite{Di_Fazio_1993} and \cite{Ragusa}. On the other hand, the counterexample in \cite{Filonov} shows that even in the divergence--free case \eqref{Div-free_drift} the smallest violation  of the critical scale can lead to the loss of continuity by weak solutions. So, in this paper we focus on results concerning drifts belonging to critical spaces.

 For the divergence-free drifts from  critical  spaces   the   best known for today results  were obtained in \cite{Fri_Vicol}, \cite{SSSZ} (see also \cite{Vicol}) where   H\" older continuity of weak solutions was proved  for the heat equation with a drift term satisfying \eqref{Div-free_drift} (see also the related result in \cite{Lis_Zhang}). In its elliptic version, these results assume that the drift  $b$ satisfying \eqref{Div-free_drift} belongs to the space  $BMO^{-1}(\Om)$, i.e. there exists a  skew-symmetric matrix  $A\in BMO(\Om)$ such that $b=\operatorname{div } A$ (see the definition of $BMO$ space at the end of this section).  If we assume $\nabla A\in L^{1,n-1}(\Om)$ (which  is in a sense   the optimal regularity for the  critical Morrey scale $L^{p,n-p}(\Om)$) then from Trudinger's imbedding \cite{Trudinger} we obtain $A\in BMO(\Om)$   and hence  H\" older continuity of weak solutions follows from  \cite{Fri_Vicol}, \cite{SSSZ}.

For the non--divergence free critical drifts satisfying the condition \eqref{Assumptions_b}   local a priori estimates of  the H\" older norms  for Lipschitz continuous solutions to the problem \eqref{Equation} were obtained earlier in \cite{NU} under the additional restriction
$b\in L^{r, n-r}(\Om)$ with $r\in (\tfrac n2,n]$. So, in our Theorem \ref{Theorem_3} we remove this   restriction. Note also that our  present  contribution  can be viewed as a multidimensional analogue of the results obtained earlier in \cite{Chernobai_Shilkin} in the 2D case.

 \medskip
Here are some comments and possible extensions:
\begin{itemize}

\item With minor changes in the  proofs our technique allows us to obtain results similar to our Theorems \ref{Theorem_1}--\ref{Theorem_3} if we replace the Laplace operator  in \eqref{Equation} by an elliptic operator in the divergence form $-\div(a(x)\nabla u)$ with a  uniformly elliptic matrix $a(x)=(a_{jk}(x))$, $a_{jk}\in L_\infty(\Om)$. We consider this  extension to be obvious.
\item In the case of the Laplace operator in \eqref{Equation} (or, more generally, in the case of the operator $-\div (a(x)\nabla u)$ with a sufficiently smooth matrix $a(x)$) our $p$-weak solutions in Theorem \ref{Theorem_1} possesses locally integrable second derivatives $u\in W^2_{s, loc}(\Om)$ with some $s>1$ (if $f$ is sufficiently regular). Hence for $p$-weak solutions the equation \eqref{Equation} is valid a.e. in $\Om$. See the related results in  \cite{Krylov_0},   \cite{Krylov_1}, \cite{Krylov_2}, \cite{Krylov_3} for the case of   elliptic operators   in the non-divergent form with non-smooth coefficients.

\item In our investigation of H\" older continuity of weak solutions in  Theorem \ref{Theorem_3} we focus on the ``optimal''  conditions for $b$ but not for $f$. It is well-known (see, for example, \cite[Theorem 5.17]{Giaquinta_ETH}) that     H\" older continuity of weak solutions to the problem \eqref{Equation} is valid if the right hand side $f$  belongs to some appropriate Morrey space. So, we believe that   the condition $f\in L_q(\Om)$ with $q>n$ in Theorem \ref{Theorem_3} can be relaxed, see related results, for example,  in \cite{Di_Fazio_1993}, \cite{Di_Fazio_2020}.
\end{itemize}

Our paper is organized as follows. In Section \ref{Auxiliary results} we introduce some auxiliary results which are basic tools of our paper. In Section
\ref{Proof_T1} we obtain  the higher integrability of weak solutions to the problem \eqref{Equation} and prove  Theorems \ref{Theorem_1}.
In Section \ref{DG_section} we introduce some De Giorgi--type classes which are convenient for the study of the equations with the coefficients from Morrey spaces. These De Giorgi classes are very similar to the classical ones (see \cite{Giaquinta_ETH}, \cite{Han_Lin},  \cite{LU}) and readers familiar with the De Giorgi technique can omit this section.  In Section \ref{Proof_T3} we show that $p$--weak solutions to the problem \eqref{Equation} belong to the above De Giorgi classes and hence they are H\" older continuous. So, in Section \ref{Proof_T3} we prove Theorem \ref{Theorem_3}. Finally, in Appendix we prove some property of the mollification operator on the class of weak Morrey spaces.

\medskip

In the paper we use the following notation. For any $a$, $b\in
\mathbb  R^n$ we denote by $a\cdot b = a_k b_k$ their  scalar product in
$\mathbb R^n$. Repeated indexes assume the summation from 1 to $n$. An index after comma means partial derivative with respect to $x_k$, i.e. $f_{,k}:=\frac{\cd f}{\cd x_k}$.  We denote by $L_p(\Omega)$ and $W^k_p(\Omega)$ the
usual Lebesgue and Sobolev spaces. We do not distinguish between functional spaces of scalar and vector functions and omit the target space in notation.  $C_0^\infty(\Om)$ is the space of
smooth functions compactly supported in $\Om$. The space
$\overset{\circ}{W}{^1_p}(\Omega)$ is the closure of
$C_0^\infty(\Omega)$ in $W^1_p(\Omega)$ norm. For $p>1$ we denote $p'=\frac p{p-1}$.  The space of
distributions on $\Om$ is denoted by $\mathcal D'(\Om)$. By  $C^\mu(\bar \Omega)$, $\mu\in (0,1)$ we denote
the spaces of   H\" older continuous functions on $\bar
\Omega$ equipped with the norm
$$
\| u\|_{C^\mu(\bar \Om)} \, := \, \sup\limits_{x\in \Om}|u(x)| \, + \, \sup\limits_{x,y\in \Om: \, x\not= y} \frac{|u(x)-u(y)|}{|x-y|^\mu}.
$$
 The symbols $\rightharpoonup$ and $\to $ stand for the
weak and strong convergence respectively. We denote by $B_R(x_0)$
the ball in $\mathbb R^n$ of radius $R$ centered at $x_0$ and write
$B_R$ if $x_0=0$. We write   $B$ instead of $B_1$. For a domain $\Om\subset \Bbb R^n$ we also denote $\Om_R(x_0):=\Om\cap B_R(x_0)$. For an open set $\Om_0\subset \Bbb R^n$ we write $\Om\Subset \Om_0$ if $\bar \Om$ is compact and $\bar \Om \subset \Om_0$.
 For $u\in L_\infty(\Om)$ we denote   $$\operatorname{osc}\limits_{\Om} u \,  := \, \esssup\limits_{\Om}u - \operatorname{essinf}\limits_{\Om}u.$$
We denote by $L^{p,\la}(\Om)$ the Morrey space equipped with the norm
$$
\| u \|_{L^{p, \la}(\Om)} \ := \ \sup\limits_{x_0\in \Om}\sup\limits_{R<\operatorname{diam}\Om}R^{-\frac \la p}\| u\|_{L_p( \Om_R(x_0))}.
$$
We denote by $\chi_\om$ the indicator function of the set $\om\subset \Bbb R^n$. For $u: \Om\to \Bbb R$ and $k\in \Bbb R$ we denote  $(u-k)_+:=\max\{ u-k, 0\}$.
We denote by $BMO(\Om)$ the space of functions with a bounded mean oscillation equipped with the semi-norm
$$
[ u]_{BMO(\Om)} \ := \  \sup\limits_{x_0\in
\Om}\sup\limits_{R<\operatorname{diam}\Om} R^{-n} \int\limits_{ \Om_R(x_0) }|u(x)-(u)_{ \Om_R(x_0) }|\, dx,
$$
where   $(f)_\om$ stands for the average of $f$ over the domain $\om\subset \Bbb R^n$. Also we  denote by $f*g$  the convolution of  functions  $f$, $g: \Bbb R^n\to \Bbb R$:
$$
(f)_\om \ := \ -\!\!\!\!\!\!\int\limits_\om f~dx \ = \ \frac
1{|\om|}~\int\limits_\om f~dx, \qquad
(f*g)(x) \ := \ \int\limits_{\Bbb R^n} f(x-y)g(y)~dy.
$$

\medskip
\noindent
{\bf Acknowledgement.} The research  of Tim Shilkin   was supported by the project 20-11027X
financed by Czech science foundation (GA\v{C}R). The research of Misha Chernobai was partially supported by Natural Sciences and Engineering Research Council of Canada (NSERC) grants RGPIN-2018-04137 and RGPIN-2023-04534.

\newpage
\section{Auxiliary results}\label{Auxiliary results}
\setcounter{equation}{0}

\bigskip

\bigskip
In this section we present several auxiliary results. For reader's convenience we formulate these results not in their full generality but only in the form they are  used in our paper.

 The first result shows that by relaxing  an exponent of the integrability of a function we can always switch  from a weak Morrey norm   to a regular one.

\medskip

\begin{proposition}\label{Holder_inequality}
  For any $p\in (1,n)$ and $1\le q<p\le r\le n$  there are positive constants $c_1$ and $c_2$ depending only on $n$, $p$ , $q$ and $r$ such that
  $$
  c_1  \, \| b\|_{ L^{q,n-q}(\Om)} \, \le \,  \| b\|_{ L^{p,n-p}_w(\Om)}\, \le \,  c_2  \, \| b\|_{ L^{r,n-r}(\Om)}
  $$
\end{proposition}

\begin{proof}
  The result follows from the H\" older inequality for Lorentz norms, see
\cite[Section 4.1]{Garfakos}.
\end{proof}

The next result shows that under the condition  \eqref{Assumptions_b} the quadratic form corresponding the drift term is non-negative.

\begin{proposition}\label{Q_Form_good_sign}   Assume  $p\ge 2$, $p'=\frac{p}{p-1}$, and $b\in L_{p'}(\Om)$   satisfies
\eqref{Assumptions_b}. Let $\mathcal B[u,\eta]$ be the bilinear form defined in \eqref{Bilinear_Form}.   Then
\begin{equation}
\forall\, u\in L_\infty(\Om)\cap \overset{\circ}{W}{^1_p}(\Om) \qquad
\mathcal B [u,u] \ \ge  \   0.
\label{Quadratic_Form_Non-negative}
\end{equation}
\end{proposition}

\begin{proof}
  The result follows from the  inclusions $ |u|^2\in \overset{\circ}{W}{^1_p}(\Om)$, $b\in L_{p'}(\Om)$, $p'=\frac p{p-1}$, and the inequality
  $$
  \mathcal B[u,u ] =  \ \frac 12\, \int\limits_\Om b\cdot \nabla |u|^2\, dx \ \ge \ 0.
  $$
\end{proof}

The next result is the general Marcinkiewicz interpolation theorem in Lorentz spaces, see \cite[Theorem 5.3.2]{Berg_Lofstrom}.

\begin{proposition}
  \label{Marcinkiewicz}
  Assume $\Om$, $\Om'\subset \Bbb R^n$ are bounded domains and assume for some $1<p_0<p_1< \infty$ a linear operator $T: L_{p_j}(\Om)\to L_{p_j}(\Om')$ is bounded for  $j=0,1$ and
  $$
  \| T\|_{L_{p_0}(\Om)\to L_{p_0}(\Om')} \, \le \, M_0, \qquad \| T\|_{L_{p_1}(\Om)\to L_{p_1}(\Om')} \, \le \, M_1
  $$
  Then for any $p\in (p_0,p_1)$ the operator $T$ is bounded from $L_{p,w}(\Om)$ to $L_{p,w}(\Om')$ and
  $$
  \| T\|_{L_{p,w}(\Om)\to L_{p,w}(\Om')} \, \le \, M_0^{1-\theta} M_1^\theta
  $$
  where $\theta\in (0,1)$  satisfies $\frac 1p = \frac{1-\theta}{p_0} + \frac \theta{p_1}$.

\end{proposition}

The next proposition proved by Chiarenza and Frasca in \cite{Chia_Fra} is the well-known extension of the result  of C.~Fefferman \cite{Fefferman} for $p=2$.
This theorem is one of basic tools in our proofs of both Theorems \ref{Theorem_1} and \ref{Theorem_3}.

\medskip
\begin{proposition}
   \label{Morrey_usual}
   Assume $p\in (1, n)$, $r\in (1, \frac np]$ and  $V \in L^{r,n-pr} (\Om)$. Then
   $$
 \int\limits_\Om |V|\, |u|^p\, dx   \ \le \ c_{n,r,p}\, \| V\|_{L^{r,n-rp}(\Om)} \| \nabla u\|_{L_p(\Om)}^p , \qquad \forall\, u \in C_0^\infty(\Om).
   $$
  with the constant $c_{n,r,p}>0$ depending only on $n$, $r$ and $p$.

\end{proposition}




\newpage
\section{Higher integrability of weak solutions}\label{Proof_T1}
\setcounter{equation}{0}

\bigskip
\medskip
In this section we present the proof of  Theorem  \ref{Theorem_1}.
First we show that in the case of   \eqref{Assumptions_b} any $p$--weak solution corresponding to a sufficiently regular  right-hand side is bounded, see the related result in \cite{Lee}.

\begin{proposition}\label{Theorem_2}  Assume  $p\ge 2$, $p'=\frac{p}{p-1}$,  $b\in L_{p'}(\Om)$   satisfies
  \eqref{Assumptions_b},   $q>n$ and $f\in L_q(\Om)$. Then any weak
solution $u\in W^1_p(\Om)$ to
  the problem \eqref{Equation} (in the sense of the integral identity \eqref{Identity})  is essentially bounded and  satisfies  the estimate
\begin{equation}
\| u\|_{L_\infty(\Om)} \ \le \ c~\| f\|_{L_q(\Om)},
\label{L_infty_Estimate}
\end{equation}
 with a
constant $c>0$ depending only on    $\Om$, $n$ and $q$.
\end{proposition}

\begin{proof}
For $m>0$   we define a truncation $T_m:\Bbb R\to \Bbb R$ by $T_m(s):= m$ for $s\ge m$, $T_m(s)=s$ for $s< m$.
For any $s\in \Bbb R$ we also denote $(s)_+:=\max\{s, 0\}$.
Now we fix some $m>0$ and denote $ \bar u := T_m(u)$.
  Then for any $k\ge 0 $ we have $$(\bar u -k)_+\in L_\infty(\Om)\cap \overset{\circ}{W}{^1_p}(\Om), \qquad \nabla (\bar u-k)_+ = \chi_{\Om [k<u< m] } \nabla u
                         $$
where $\Om [k<u<m]=\{\, x\in\Om: \, k<u(x)<m \, \}$.  Approximating  $\eta := (\bar u -k)_+$ by smooth functions we can  take $\eta$ as a test function  in \eqref{Identity}. For $k\ge m$ we have $\eta\equiv 0$ and hence $\mathcal B[u, \eta]=0$. For $k< m$ the condition \eqref{Assumptions_b} implies
  $$
  \gathered
  \mathcal B[u, \eta] \ = \   \tfrac 12\, \int\limits_{  \Om[k<  u<  m] } b\cdot \nabla | \eta |^2\, dx   \   + \ (m-k)\, \int\limits_{\Om[ u\ge m]  } b\cdot \nabla u  \, dx
    \\ =
   \mathcal B[\eta , \eta] \ + \ (m-k)\, \int\limits_{\Om   } b\cdot \nabla (u-m)_+ \, dx \  \ge \ \mathcal B[\eta , \eta]
   \ \ge \ 0
  \endgathered
  $$
 Hence for $A_k:=\{ \, x\in \Om: \, \bar u(x)>k \, \}$  from   \eqref{Identity} we obtain
 $$
 \int\limits_\Om |\nabla (\bar u-k)_+|^2\, dx \ \le \ \| f\|_{L_q(\Om)}^2 |A_k|^{1-\frac 2q}, \qquad \forall\, k\ge 0,
 $$
 which implies (see   \cite[Chapter II, Lemma 5.3]{LU})
 $$
 \esssup\limits_\Om \bar u \ \le \ c \, |\Om|^{ \dl } \, \| f\|_{L_q(\Om)}, \qquad \dl:= \tfrac 1n-\tfrac 1q,
 $$
 with some constant $c>0$ depending only on $n$ and $q$.
 As this estimate is uniform with respect to $m>0$ we conclude  $u$ is essentially bounded from above and
 $$
 \esssup\limits_\Om  u \ \le \ c \, |\Om|^{ \dl } \, \| f\|_{L_q(\Om)}.
 $$
 Applying the same procedure to $\bar u:= T_m(-u)$  instead of $\bar u:= T_m(u)$  we obtain $u\in L_\infty(\Om)$ as well as \eqref{L_infty_Estimate}.
\end{proof}

Proposition \ref{Theorem_2} implies the uniqueness theorem for the problem \eqref{Equation} in the class of $p$-weak solutions.

\begin{proposition}\label{Uniqueness} Assume  $p\ge 2$, $p'=\frac{p}{p-1}$, and $b\in L_{p'}(\Om)$   satisfies
\eqref{Assumptions_b}.   Then for any $f\in L_p(\Om)$
  the problem \eqref{Equation} can have at most one weak solution in the class $W^1_p(\Om)$.
\end{proposition}

\begin{proof}
  Assume $u_1\in W^1_p(\Om)$ and  $u_2\in W^1_p(\Om)$ are two weak solutions of \eqref{Equation} corresponding to the same $f\in L_p(\Om)$. Then $u:=u_1-u_2$ is a weak solution of \eqref{Equation} corresponding to $f\equiv 0$. From Proposition \ref{Theorem_2} we obtain $u\equiv 0$ in $\Om$.
\end{proof}

The next proposition is the energy estimate. Note that the constant $c>0$ in this estimate  does not depend on the drift $b$ satisfying the condition \eqref{Assumptions_b}.

\begin{proposition}\label{L_2_estimate}    Assume  $p\ge 2$, $p'=\frac{p}{p-1}$,   $b\in L_{p'}(\Om)$   satisfies
\eqref{Assumptions_b}, and $f\in L_q(\Om)$ with $q>n$. Then for any  weak
solution  $u\in W^1_p(\Om)$   to
  the problem \eqref{Equation} the following estimate hold:
$$
\| u\|_{W^1_2(\Om)}  \ \le \ c~\| f\|_{L_2(\Om)}
$$
with some constant $c>0$ depending only on $n$  and $\Om$.
\end{proposition}

\begin{proof}
 From Proposition \ref{Theorem_2} we obtain $u\in L_\infty(\Om)$ and hence the test function $\eta := u\in L_\infty(\Om)\cap \overset{\circ}{W}{^1_p}(\Om)$ is admissible for \eqref{Identity}. From Proposition \ref{Q_Form_good_sign} we obtain $\mathcal B[u,u]\ge 0$ and hence  from \eqref{Identity} we obtain the inequality
 $$
 \|\nabla u\|_{L_2(\Om)}^2 \ \le \ \int\limits_\Om f\cdot \nabla u\, dx
 $$
 which the result follows from.
\end{proof}

Our next result is the estimate of the bilinear form corresponding to the drift term. This result is a direct consequence of the Feffermann--Chiarenza--Frasca estimate, see Proposition \ref{Morrey_usual}.

\medskip
\begin{proposition}
   \label{Morrey_Compactness}
   Assume  $r\in \left( \frac {2n}{n+2}, 2\right) $ and  $b \in L^{r, n-r} (\Om)$. Then  there exists  $c>0$ depending only on $n$ and  $r$  such that    for any $u \in W^1_2(\Om)$ and    $\zeta \in C_0^\infty(\Om)$  the following estimate   holds:
   \begin{equation}\label{Morrey_higher_integrability}
  \int\limits_\Om \zeta^3 \, |b|\, |u |^2\, dx \ \le \ c\, \| b\|_{L^{r, n-r }(\Om)}  \,   \|\nabla (\zeta^2 u)\|_{L_2(\Om)} \,   \|\nabla (\zeta u)\|_{L_{\frac {2n}{n+2}}(\Om)}.
   \end{equation}
   Moreover, for any   $\theta\in \left(\frac nr -\frac n2,1\right)$   there exists  $c>0$ depending only on $n$, $r$ and $\theta$  such that    for any $u \in W^1_2(\Om)$ and  any $\zeta \in C_0^\infty(\Om)$  satisfying $0\le \zeta \le 1$ we have
  \begin{equation}\label{Morrey_Compactness_estimate}
 \int\limits_\Om \zeta^{1+\theta} \, |b|\, |u |^2\, dx \ \le \ c \, \| b\|_{L^{r, n-r }(\Om)} \|   u \|_{L_2(\Om)}^{1-\theta}   \|\nabla (\zeta u)\|_{L_2(\Om)}^{1+\theta}\, |\Om|^{\theta/n}.
  \end{equation}
  If we additionally assume $u\in \overset{\circ}{W}{^1_2}(\Om)$ then the estimates \eqref{Morrey_higher_integrability} and \eqref{Morrey_Compactness_estimate} remains true for an arbitrary $\zeta \in C_0^\infty(\Bbb R^n)$.
\end{proposition}

\begin{proof} To prove \eqref{Morrey_higher_integrability} we apply the H\" older inequality
$$
\int\limits_\Om \zeta^3 \, |b|\, |u |^2\, dx \ \le \  c\, \|  \zeta^2 u \|_{L_{\frac{2n}{n-2}}(\Om)}  \|  \zeta b u \|_{L_{\frac{2n}{n+2}}(\Om)}.
$$
Denote $q:=\frac{2n}{n+2}$,  $V:= |b|^q$, $r_1:= \frac r{q}$, $r_1>1$. Then $$V\in L^{r_1, n-qr_1}(\Om), \qquad \| V\|_{L^{r_1, n-qr_1}(\Om)}\, = \,  \| b\|^q_{L^{r, n-r}(\Om)}$$ and from Proposition \ref{Morrey_usual} we obtain
$$
\|  \zeta b u \|_{L_q(\Om)}  \ \le \ c\,   \| V\|_{L^{r_1, n-qr_1}(\Om)}^{1/q}  \|  \nabla (\zeta u) \|_{L_q(\Om)} \ \le \ c\,  \| b\|_{L^{ r, n- r}(\Om)} \,   \|  \nabla (\zeta u) \|_{L_q(\Om)}.
$$
On the other hand, by the Gagliardo-Nirenberg inequality we obtain
$$
\|  \zeta^2 u \|_{L_{\frac{2n}{n-2}}(\Om)} \, \le \, c\,   \| \nabla( \zeta^2 u) \|_{L_2(\Om)}
$$
and we arrive at \eqref{Morrey_higher_integrability}.

Now we turn to the proof of \eqref{Morrey_Compactness_estimate}. Take   $q\in \left(   \frac{2n}{n+2}, r\right)$ such that    $$\tfrac{1}{q'}=\tfrac {1-\theta}2 +  \theta \,  \tfrac{n-2}{2n}, \qquad q'=\tfrac q{q-1}.$$ Then
$$
\int\limits_\Om \zeta^{1+\theta} \, |b|\, |u |^2\, dx \ \le \  c\, \|  \zeta^\theta u \|_{L_{q'}(\Om)}  \|  \zeta b u \|_{L_q(\Om)}.
$$
Denote $V:= |b|^q$,  $r_1:= \frac r{q}$, $r_1>1$. Then $V\in L^{r_1, n-qr_1}(\Om)$  and
$$
\|  \zeta b u \|_{L_q(\Om)}  \ \le \ c\,   \| V\|_{L^{r_1, n-qr_1}(\Om)}^{1/q}  \|  \nabla (\zeta u) \|_{L_q(\Om)} \ \le \ c\,  \| b\|_{L^{ r, n- r}(\Om)} \,   \|  \nabla (\zeta u) \|_{L_q(\Om)}
$$
On the other hand,  we can interpolate
$$
\gathered
 \|  \zeta^\theta u \|_{L_{q'}(\Om)} \,=\|  (\zeta |u|)^\theta |u|^{1-\theta} \|_{L_{q'}(\Om)}\, \le
\\
\le \, \| (\zeta |u|)^\theta\|_{L^{\frac{2n}{(n-2)\theta}}(\Om)} \||u|^{1-\theta} \|_{L_{\frac{2}{1-\theta}}(\Om)}  \le  \, c \, \|  \nabla (\zeta u) \|_{L_2(\Om)}^{\theta } \|u \|_{L_{2}(\Om)}^{1-\theta }
\endgathered
$$
with some constant $c>0$ depending only on $n$.
So, we obtain
$$
 \int\limits_\Om \zeta^{1+\theta} \, |b|\, |u |^2\, dx \ \le \ c\, \|    u \|_{L_2(\Om)}^{1-\theta }   \|\nabla (\zeta u)\|_{L_2(\Om)}^{\theta }\,
     \| b\|_{L^{ r, n- r}(\Om)} \,   \|  \nabla (\zeta u) \|_{L_q(\Om)}.
$$
Using the H\" older inequality we obtain
$$
 \|  \nabla (\zeta u) \|_{L_q(\Om)} \ \le \  \|  \nabla (\zeta u) \|_{L_2(\Om)} |\Om|^{\frac 1q-\frac 12} \ = \  \|  \nabla (\zeta u) \|_{L_2(\Om)} |\Om|^{\frac \theta {n}}
$$
which completes  the proof.
\end{proof}

The next proposition is the higher integrability estimate  which is the main result of the present section. In this proposition we assume that   $b$ is   smooth and we derive \eqref{Main_Estimate} as an a priori estimate  for a sufficiently regular drift.
\begin{proposition}\label{L_p_estimate}   Assume   $b\in C^\infty(\bar \Om)$  satisfies \eqref{Assumptions_b} and assume $f\in L_q(\Om)$ with $q>n$.  Let
  $u\in \overset{\circ}{W^1_2}(\Om)\cap L_\infty(\Om)$ satisfy \eqref{Equation} in terms of distributions (i.e. the identity \eqref{Identity} holds).
Then   there exists $p>2$ depending only on  $n$,  $q$ , $\Om$  and   $\| b\|_{L^{2, n-2}_w(\Om)}$   such
that   $u\in \overset{\circ}{W^1_p}(\Om)$ and
the  estimate \eqref{Main_Estimate} holds
with some constant $c>0$ depending only on $n$,   $q$, $\| b\|_{L^{2, n-2}_w(\Om)}$ and the Lipschitz constant of $\cd \Om$.
\end{proposition}

\begin{proof} Let us fix some
 $r\in (\frac{2n}{n+2}, 2)$.
   First we consider  an internal point $x_0\in \Om$. Take
  $R>0$ such that $B_{2R}(x_0)\subset \Om$ and choose a cut-off
function $\zeta\in C_0^\infty(B_{2R}(x_0))$ so  that $\zeta\equiv 1$
on $B_R(x_0)$ and $|\nabla\zeta |\le c/R$.
Denote $\bar u = u-(u)_{B_{2R}(x_0)}$. Then the function  $\eta= \zeta^4\bar u$ is admissible for   \eqref{Identity}. It is easy to see that
  $$
\mathcal B [u,\zeta^4\bar u] \ = \    \mathcal B [\zeta^2 \bar
u,\zeta^2 \bar u]  \ - \  2\, \int\limits_{B_{2R}(x_0)} \zeta^3
|\bar u |^2 \, b  \cdot\nabla \zeta \, dx.
$$
  Taking into account   \eqref{Quadratic_Form_Non-negative} from \eqref{Identity} we obtain
$$
  \| \zeta^2 \nabla   u\|_{L_2(B_{2R}(x_0))}^2 \ \le \ c\, \Big( \| \bar u \nabla \zeta  \|_{L_2(B_{2R}(x_0))}^2 + \| f \|_{L_2(B_{2R}(x_0))}^2\Big)   \ + \ c \, \int\limits_{B_{2R}(x_0)} \zeta^3
|\bar u |^2 \,  b  \cdot \nabla \zeta~dx.
$$
We estimate the last term with the help of \eqref{Morrey_higher_integrability}:
$$
\gathered
\int\limits_{B_{2R}(x_0)} \zeta^3
|\bar u |^2 \,  b  \cdot \nabla \zeta~dx \ \le \\ \le \   \|\nabla \zeta\|_{L_\infty(B_{2R}(x_0))} \, \| b\|_{L^{r,n-r}(\Om)}\,  \|\nabla (\zeta^2 \bar u) \|_{L_2(B_{2R}(x_0))} \,  \| \nabla (\zeta\bar u) \|_{L_{\frac{2n}{n+2}}(B_{2R}(x_0))}.
\endgathered
$$
Using the Young inequality we get
\begin{equation}
\gathered
  \| \zeta^2 \nabla   u\|_{L_2(B_{2R}(x_0))}^2 \ \le \ c\, \Big( \| \bar u \nabla \zeta  \|_{L_2(B_{2R}(x_0))}^2 + \| f \|_{L_2(B_{2R}(x_0))}^2\Big)   \ + \\ + \
c\, \|\nabla \zeta\|_{L_\infty(B_{2R}(x_0))}^2 \,   \| b\|_{L^{r,n-r}(\Om)}^2 \,
 \|  \nabla (\zeta \bar u) \|_{L_{\frac{2n}{n+2}}(B_{2R}(x_0))}^2.
 \endgathered
\label{Inequality_for_u_bar}
\end{equation}
Taking into account $\zeta\equiv 1$ on $B_R(x_0)$ and  $\|\nabla \zeta\|_{L_\infty(B_{2R}(x_0))} \le \frac cR$ and the inequalities
\begin{equation}
  \| \bar u    \|_{L_2(B_{2R}(x_0))} \, \le \, c\,   \|  \nabla   u  \|_{L_{\frac{2n}{n+2}}(B_{2R}(x_0))}, \label{Poincare}
\end{equation}
$$
 \|  \nabla (\zeta \bar u) \|_{L_{\frac{2n}{n+2}}(B_{2R}(x_0))} \, \le \, c\,  \|  \nabla   u  \|_{L_{\frac{2n}{n+2}}(B_{2R}(x_0))},
$$
we arrive at
 $$
  \|   \nabla   u\|_{L_2(B_{ R}(x_0))}  \ \le \     c\, \| f \|_{L_2(B_{2R}(x_0))}    \ + \
\frac c{R} \, \Big( 1+  \| b\|_{L^{r, n-r}(\Om)} \Big) \,
 \|   \nabla  u \|_{L_{\frac{2n}{n+2}}(B_{2R}(x_0))}.
$$
Here $c$ is a global constant that depends only on $r$. Next we derive similar estimates near the boundary.
Extend $u$ and $f$  by zero outside $\Om$ and denote this extension
by $\tilde u$ and $\tilde f$. Note that as $\cd \Om$ is
Lipschitz there exist $R_0>0$ and $\dl_0>0$ depending on $\Om$ such that
\begin{equation}\label{Oblast'}
\forall\, R<R_0, \qquad \forall\, x_0\in \cd \Om \qquad |B_R(x_0) \setminus \Om | \, \ge \,   \dl_0 |B_R|.
\end{equation}
  Take $x_0\in \cd \Om$,
$R<R_0/2$ and denote $\Om_R(x_0):= B_R(x_0)\cap \Om$. Choose a cut-off function $\zeta\in
C_0^\infty(B_{2R}(x_0))$ so that $\zeta\equiv 1$ on $B_R(x_0)$. Then the function
$$
\eta \ := \ \zeta^4 u \ \in \   L_\infty(\Om_{2R}(x_0))\cap  \overset{\circ}{W}{^1_2}(\Om_{2R}(x_0))$$  is admissible for the identity
   \eqref{Identity}.
Proceeding similar to the internal case we arrive at the inequality \eqref{Inequality_for_u_bar} with $\bar u$ replaces by $\tilde u$. Then instead of \eqref{Poincare} we use the inequality
$$
 \| \tilde  u    \|_{L_2(B_{2R}(x_0))} \, \le \, c\,   \|  \nabla   \tilde u  \|_{L_{\frac{2n}{n+2}}(B_{2R}(x_0))},
$$
which holds with some constant $c>0$ depending only on $n$ and $\dl_0$ in \eqref{Oblast'} (and independent on $x_0$ and $R$) as $\tilde u$ satisfies $|\{ \, x\in B_{2R}(x_0): \, \tilde u (x)=0\, \}|\ge \dl_0|B_{2R}|$. Then
   after
routine computations     we obtain for any $x_0\in \cd \Om$ and $R<R_0$
 $$
  \|   \nabla   u\|_{L_2(\Om_{ R}(x_0))}  \ \le \     c\, \| f \|_{L_2(\Om_{2R}(x_0))}    \ + \
\frac c{R} \, \Big( 1+  \| b\|_{L^{r, n-r}(\Om)} \Big) \,
 \|   \nabla  u \|_{L_{\frac{2n}{n+2}}(\Om_{2R}(x_0))}.
$$
Combining internal and boundary estimates in the standard way and dividing the result by $R^{n/2}$ for any $x_0\in \bar \Om$ and $R<R_0$ we obtain the reverse H\" older  inequality for $\nabla u$:
$$
\gathered \left(  \ \ \  \  -  \!\!\!\!\!\!\!\!\!\!\!
\int\limits_{\Om_R(x_0)}|\nabla u|^2~dx\right)^{\frac 12} \ \le \
 c_1\,
\left(  \ \ \  \  -  \!\!\!\!\!\!\!\!\!\!\!\!
\int\limits_{\Om_{2R}(x_0)}|f|^2~dx\right)^{\frac  12} \ +  \\ +  \
c_2\, \Big(1+ \|b  \|_{L^{r, n-r}(\Om)}  \Big) \left( \ \ \  \ -
 \!\!\!\!\!\!\!\!\!\!\!\! \int\limits_{\Om_{2R}(x_0)}|\nabla
u|^{\frac{2n}{n+2}}~dx\right)^{\frac  {n+2}{2n}}.
\endgathered
$$
From Proposition \ref{Holder_inequality} we conclude
$$
\|b  \|_{L^{r, n-r}(\Om)} \, \le \, c\, \|b  \|_{L^{2, n-2}_w(\Om)}.
$$
Now the higher integrability of $u$ follows via Gehring's lemma,  see, for example, \cite[Chapter V]{Giaquinta}. Namely,
there exists $p>2$  and $c_*>0$ depending only on $c_1$, $c_2$ and  $\|b  \|_{L^{2, n-2}_w(\Om)}$  such that   $\nabla u\in L_p(\Om)$ and the
following estimate holds:
$$
\| \nabla u\|_{L_p(\Om)} \ \le \ c_*~(\| \nabla u\|_{L_2(\Om)} + \|
f\|_{L_p(\Om)}).
$$
Combining this estimate with Proposition \ref{L_2_estimate} we obtain \eqref{Main_Estimate}.
\end{proof}

Finally we can prove Theorem \ref{Theorem_1}.

\begin{proof} We follow the method similar to \cite[Theorem 2.1]{Kwon_1},  and \cite{Verchota_1},\cite{Verchota_2}. Let us fix some $q>n$. Let  $p>2$  be the  exponent defined in Proposition \ref{L_p_estimate}  and denote $p'=\frac p{p-1}$.
As $\Om$ is Lipschitz we can find  the sequence of $C^2$-smooth domains
$\{ \Om_k\}_{k=1}^\infty$   such that $$\Om_{k+1}\Subset \Om_k, \qquad \bigcup\limits_k \Om_k = \Om.$$
Moreover,  it is possible to  construct   domains $\Om_k$ so that  the Lipschitz constants of $\cd \Om_k$ are controlled uniformly by the Lipschitz constant of $\cd \Om$. In particular, we can assume there are exist positive constants $\hat \dl_0$ and $\hat R_0$ independent on $k\in \Bbb N$ such that
\begin{equation}\label{Oblast'_k}
\forall\, k\in \Bbb N, \qquad  \forall\, R<\hat R_0, \qquad \forall\, x_0\in \cd \Om_k \qquad |B_R(x_0) \setminus \Om_k | \, \ge \,   \hat  \dl_0 \, |B_R|.
\end{equation}
For a canonical domain given as a subgraph of a Lipschitz function existence of such approximation can be obtained by mollification and shift of the graph, and for general bounded Lipschitz domain  the standard localization works.

Now we take a sequence of positive numbers $\ep_k\to 0$ such that   $\ep_k < \operatorname{dist}\{ \bar \Om_k, \cd \Om\}$, and define the mollification of the drift $b$:
  $$
  b_k (x )  \, := \, \int\limits_{\Om} \om_{\ep_k}(x-y)b(y)\, dy, \qquad x\in \Om,
$$
where $\om_\ep(x):=\ep^{-n}\om(x/\ep)$ and $\om\in C_0^\infty(\Bbb R^n)$ is the standard Sobolev kernel, i.e.
\begin{equation}
\om \, \ge \, 0, \qquad \supp \om \in \bar B, \qquad \int\limits_{\Bbb R^n}\om(x)\, dx \, = \, 1, \qquad \om(x)=\om_0(|x|).
\label{Sobolev_kernel}
\end{equation}
Then $b_k \in C^\infty(\bar \Om)$ and as $\ep_k<\operatorname{dist}\{ \bar \Om_k, \cd \Om\}$ from  \eqref{Assumptions_b}  for any $k $ we obtain
\begin{equation}\label{Assumptions_b_k}
\div  b_k \, \le \, 0 \qquad\mbox{in} \quad \Om_k.
\end{equation}
Moreover, there is a constant $c>0$ independent on $k$ such that
\begin{equation}\label{Drift_approximation1}
\| b_k\|_{L^{2,n-2}_w(\Om_k)} \, \le \, c\, \|b\|_{L^{2,n-2}_w(\Om)}, \qquad \| b_k -b\|_{L_{p'}(\Om)} \to 0 \quad \mbox{as} \quad k \to 0.
\end{equation}
For the reader's convenience
we prove the inequality \eqref{Drift_approximation1} in Appendix.

For  $f\in L_p(\Om)$ we can find    $f_k\in C^\infty_0(\Om_k)$ such that $\|f_k-f\|_{L_p(\Om)}\to 0$.
  Let $u_k\in W^1_2(\Om)\cap L_\infty(\Om)$ be  smooth  solutions to the regularized problem
  \begin{equation}\label{Equation_Approx}
\left\{ \quad \gathered     -\Delta u_k +  b_k \cdot \nabla u_k \ = \ -\div f_k \qquad\mbox{in}\quad \Om_k, \\
u_k|_{\cd \Om_k} \ = \ 0,
\endgathered\right.
\end{equation}
and extend functions $u_k$ by zero from $\Om_k$ onto $\Om$. As the drift $b_k$ of in \eqref{Equation_Approx} satisfies the properties \eqref{Assumptions_b_k} and \eqref{Drift_approximation1}
from Proposition \ref{L_p_estimate} we obtain the estimate
$$
\| u_k \|_{W^1_p(\Om)} \ \le \ c\, \| f_k\|_{L_p(\Om)}
$$
with a constant  $c>0$ depending only on $n$, $q$,  $\| b\|_{L^{2, n-2}_w(\Om)}$ and the constant $\hat \dl_0$ in \eqref{Oblast'_k} which is independent on $k$.  Hence we can take a subsequence $u_k$ such that
$$
u_k \rightharpoonup u \quad \mbox{in}\quad W^1_p(\Om).
$$
Take any $\eta\in C_0^\infty(\Om)$, due to our construction of the sequence $\Omega_k$ we can find $N$ such that for all $k\ge N,~\supp \eta\subset\Om_k$, therefore $\eta$ is a suitable test function in \eqref{Equation_Approx}. As $b_k \to b$ in $L_{p'}(\Om)$ we can pass to the limit in the  identity
$$
\gathered  \int\limits_{\Om}   \nabla u_k \cdot\Big(
\nabla\eta + b_k \eta \Big) \, dx   \ = \
\int\limits_{\Om} f_k \cdot \nabla \eta~dx   ,
\endgathered
\label{Identity_eps}
$$
and obtain \eqref{Identity}. Hence $u \in \overset{\circ}{W}{^1_p}(\Om)$ is a $p$-weak solution to the problem \eqref{Equation}. Its uniqueness follows from Proposition \ref{Uniqueness}.
\end{proof}

\newpage
\section{Modified De Giorgi classes}\label{DG_section}
\setcounter{equation}{0}

\bigskip
\medskip
In this section we introduce the modified De Giorgi classes which are convenient for the study of solutions to the elliptic equations with coefficients from Morrey spaces.  The only difference of our classes from  the standard De Giorgi classes (see, for example  \cite[\S 8.3]{Giaquinta_ETH}, \cite[\S II.6]{LU}) is the additional multiplier $1+ \frac{R^\al}{(R-\rho)^\al}$ in the inequalities \eqref{DG} and \eqref{DG_boundary}.
\begin{definition}\label{Def_DG}
  For $u\in W^1_2(\Omega)$ we define $A_k:= \{\, x\in \Om: \, u(x)>k\, \}$. We say $u\in DG(\Om, k_0)$ if there exist  constants $\ga>0$, $F>0$, $\al\ge 0$, $q>n$ such that for any $B_R(x_0)\subset \Om$, any $0<\rho <R $ and any $k\ge k_0$ the following inequality holds
\begin{equation}
\gathered
  \int\limits_{A_k\cap B_\rho(x_0)} |\nabla u|^2\, dx \ \le \ \frac{\ga^2}{(R-\rho)^2}\left( 1+  \frac{R^\al}{(R-\rho)^\al}\right) \, \int\limits_{A_k\cap B_R(x_0)}|u-k|^2\,dx \ + \\ + \  F^2 \, |A_k\cap B_R(x_0)|^{1-\frac 2q} \label{DG}.
\endgathered
\end{equation}
   We write $u\in DG(\Om)$ if  $u\in DG(\Om, k_0)$ for any $k_0\in \Bbb R $. To avoid overloaded notation,  when  we need to specify constants in Definition \ref{Def_DG} we allow some terminological license and say that  the class $D(\Om,k_0)$ corresponds to the constants $\ga$, $F$, $\al$, $q$  instead of including these constants  in the notation of the functional class.
\end{definition}

The second modified De Giorgi class characterizes the behavior of  Sobolev functions  near the boundary of the domain:

\begin{definition}\label{Def_DG_boundary}
  We say $u\in DG(\cd\Om)$ if $u\in   \overset{\circ}{W}{^1_2}(\Om)$ and there exist  constants $\ga>0$, $F>0$, $\al\ge 0$, $q>n$, $R_0>0$ such that for any $x_0\in \cd\Om$, any $0<\rho <R \le R_0$ and any $k\ge 0$ the following inequality holds
\begin{equation}
\gathered
  \int\limits_{A_k\cap \Om_\rho(x_0)} |\nabla u|^2\, dx \ \le \ \frac{\ga^2}{(R-\rho)^2}\left( 1+  \frac{R^\al}{(R-\rho)^\al}\right) \, \int\limits_{A_k\cap \Om_R(x_0)}|u-k|^2\,dx \ + \\ + \ F^2 \, |A_k\cap \Om_R(x_0)|^{1-\frac 2q} \label{DG_boundary}
\endgathered
\end{equation}
  where we denote  $\Om_R(x_0):=\Om\cap B_R(x_0)$.
\end{definition}

 We emphasize that while $u\in DG(\Om)$ implies that \eqref{DG} holds for any $k\in \Bbb R$ the inclusion $u\in DG(\cd \Om)$ implies that \eqref{DG_boundary}  holds only for non-negative $k\in \Bbb R$. Such choice of the boundary De Giorgi class  is motivated by the homogeneous Dirichlet boundary conditions in \eqref{Equation}.

\medskip
The main result of this section is the following theorem:

\begin{theorem}\label{DG_Theorem} Assume $\Om \subset \Bbb R^n$ is a bounded Lipschitz  domain and denote by $DG(\Om)$ and $DG(\cd \Om)$ the modified De Giorgi classes corresponding to some positive  constants $F$, $\ga$, $\al$ and  $q>n$.
\begin{itemize}
  \item[{\rm (i)}] Assume  $\pm u \in DG(\Om) $. Then there exists $\mu\in (0,1)$   depending only on $n$, $q$, $\ga$  such that $u\in C^\mu_{loc}(  \Om)$  and for any $\Om'\Subset \Om$
  $$
  \| u\|_{C^\mu(\bar \Om')} \ \le \ c(\Om, \Om')\, \Big( \| u\|_{L_2(\Om)} +  F\Big)
  $$
  with some constant $c(\Om, \Om')>0$ depending only on $\Om$, $\Om'$, $n$, $q$, $\ga$ and $\al$.
  \item[{\rm (ii)}] Assume   $\pm u \in DG(\Om)\cap DG(\cd \Om)$. Then there exists $\mu\in (0,1)$   depending only on $n$, $q$, $\ga$  such that $u\in C^\mu(  \bar \Om)$ and
   $$
  \| u\|_{C^\mu(\bar \Om)} \ \le \ c(\Om)\, \Big( \| u\|_{L_2(\Om)} +  F\Big)
  $$
  with some constant $c(\Om)>0$ depending only on $\Om$, $R_0$,  $n$, $q$, $\ga$ and $\al$.
\end{itemize}

\end{theorem}

Our  modified De Giorgi classes  in Definitions \ref{DG} and \ref{DG_boundary} are very similar to ones introduced earlier in \cite{LU}.
Our proof of Theorem \ref{DG_Theorem} is  standard and follows to the general line in \cite{LU}. But due to the minor difference in the proof and for the reader convenience we present the proof of Theorem \ref{DG_Theorem} in this section.

\medskip

First we present some technical lemma.
\begin{lemma}\label{recurrsion}
Assume  $\{Y_m\}_{m=1}^\infty\subset \Bbb R$ are nonnegative and there exist $\ep>0$, $\ep_1\ge 0$,
$b\ge 1$ and $c_0\ge 1 $ such that the following inequality holds
\begin{equation}\label{Fix_constant}
Y_{m+1} \ \le \ c_0 ~ b^m~
Y_m^{1+\ep} (1+ Y_m^{\ep_1}), \qquad \forall~m\in \Bbb N.
\end{equation}
Assume $Y_0  \le  \theta$ where
 $\theta \ := \ (2c_0)^{-1/\ep} ~b^{-1/\ep^2} $. Then
$$
 Y_m \ \le \ \theta
~b^{-m/\ep}, \qquad \forall~m\in \Bbb N.
$$
\end{lemma}
\begin{proof} We use the induction with respect to $m$. For $m=0$ the statement holds by assumption. Now assuming $ Y_m \ \le \ \theta
~b^{-m/\ep}$ from \eqref{Fix_constant} we estimate
$$
\gathered
Y_{m+1} \ \le \ c_0 ~ b^m~
Y_m^{1+\ep} (1+ Y_m^{\ep_1}) \ \le \  c_0 ~ b^m~ (\theta
~b^{-m/\ep})^{1+\ep} \Big( 1+ (\theta
~b^{-m/\ep})^{\ep_1} \Big) \ \le  \\ \le \ 2\, c_0 \, \theta^{1+\ep} \, b^{-\frac m\ep} \ = \ 2\, c_0 \, \theta^{1+\ep} \, b^{1/\ep}\, b^{-\frac {m+1}\ep}
\endgathered
$$
To finish the proof we note that $\theta    =   (2c_0)^{-1/\ep} ~b^{-1/\ep^2}$ implies \,  $ 2  c_0 \, \theta^{1+\ep} \, b^{1/\ep} = \theta$. \
\end{proof}

Now we prove the local maximum estimate:
\begin{lemma}\label{Boundedness_Holder}
Assume  $u\in DG(B_R, k_0)$. Then $u$ is essentially bounded from above in $B_{R/2}$ and
\begin{equation}\label{Boundedness_Estimate_Holder}
\sup\limits_{B_{R/2} } ~(u-k_0)_+ \ \le \ c_*  \,  \left[\Big(\
-\!\!\!\!\!\!\!\int\limits_{B_R } |(u-k_0)_+|^2~dx \Big)^{1/2} + F\, R^{1-\frac nq}\right]
\end{equation}
where $(u-k_0)_+:=\max\{ u-k_0,0\}$ and  $c_*>0$ depends only on $n$, $\ga$ and $\al$.
\end{lemma}

\begin{proof}
Without loss of generality we can assume $R=1$ and $k_0=0$.
Denote
$$
H \ : = \   c_* \left[ \Big(\ -\!\!\!\!\!\!\int\limits_{B} u_+^2~dx  \
\Big)^{1/2} +F~\right],
$$
where the precise value of the constant $c_*=c_*(n,\ga, \al)>0$ will be fixed later.
Define
$$
\gathered \rho_m:=\tfrac 12+\tfrac{1}{2^{m+1}}, \quad
k_m:=H-\tfrac{H}{2^m},  \quad J_m:=\int\limits_{B_{\rho_m}} (u - k_m)^2_+~dx  , \\  \mu_m \, := \,  |\{ \, x\in B_{\rho_m}: \, u(x)>k_{m+1}\, \}|.
\endgathered
$$
From
$  J_m  \ge   (k_{m+1} - k_{m})^2 \mu_{m}
$
we obtain
\begin{equation}
 \mu_{m} \, \le \, \frac{2^{2m+2}}{H^2}~J_{m}.
\label{(M)}
\end{equation}
Define $\rho_m':=\frac 12(\rho_{m+1}+\rho_m)$ and define cut-off functions  $$\zeta_m\in C^\infty_0(
B_{\rho_m'}), \qquad
\zeta_m\equiv 1 \quad\mbox{on} \quad B_{\rho_{m+1}} , \qquad
|\nabla\zeta_m|\le \frac {c}{\rho_m-\rho_{m+1}}.
$$
From the  H\" older  inequality  and the imbedding $W^1_2(B_{\rho_m'})\hookrightarrow L_{\frac{2n}{n-2}(B_{\rho_m'})}$ we obtain
\begin{equation}\label{(J)}
J_{m+1}  \ \le \ c(n)~\mu_m ^{\frac{2}{n}}~ \Big( ~ \| \nabla
(u-k_{m+1})_+\|_{L_{2}(B_{\rho_m'})}^2
 \ + \ \frac
{c(n)}{(\rho_m -\rho_{m+1})^2}~J_m~\Big).
\end{equation}
Using    $u\in DG(B)$ we estimate
\begin{equation}\label{(DJ)}
\|  \nabla (u-k_{m+1})_+\|_{L_{2}(B_{\rho_m}')}^2  \ \le \  \frac{\ga^2}{(\rho_{m}-\rho_{m+1})^2}\left( 1+  \frac{\rho_{m}^\al}{(\rho_{m}-\rho_{m+1})^\al}\right) \, J_m \ +   \ F^2\mu_m^{1-\frac 2q}
\end{equation}
and hence
$$
J_{m+1}  \ \le \ \frac
{c(n,\ga)}{(\rho_m -\rho_{m+1})^2}~\mu_m ^{\frac{2}{n}}~ \left( 1+  \frac{\rho_{m}^\al}{(\rho_{m}-\rho_{m+1})^\al}\right)      ~J_m \ +   \ F^2\mu_m^{1+\frac 2n -\frac 2q}.
$$
Denote $$Y_m   \, := \, \frac {J_m}{H^2}.$$
Taking into account \ $\rho_m - \rho_{m+1}= \frac {1}{2^{m+2}}$,  $\rho_m\le 1$  from \eqref{(M)}, \eqref{(J)}, \eqref{(DJ)}
we obtain
$$
Y_{m+1} \ \le \ c\, 2^{2m (1+\frac 2n)}\left(1+ 2^{\al m}\right)\, Y_m^{1+\frac 2n} \ + \ c\, \tfrac{F^2}{H^2}\, 2^{m \left(1+\frac 2n -\frac 2q\right) } Y_m^{1+\frac 2n -\frac 2q}.
$$
Using and $\frac FH\le 1$  we arrive at \eqref{Fix_constant}
\begin{equation}
Y_{m+1} \ \le \ c_0 ~ b^m~
Y_m^{1+\ep} (1+ Y_m^{\ep_1}),
\end{equation}
where  $b:= 2^{ 2+\al+\frac{4}{n } }$,  $\ep=\frac 2n -\frac 2q$, $\ep_1=\frac2q$ and $c_0$ depends only on $n$, $\ga$, $\al$. So, the sequence $\{ Y_m\}$ satisfies the assumptions of Lemma \ref{recurrsion}.
Now fix $c_*> \theta^{-1/2}$ where $\theta $ is defined in Lemma \ref{recurrsion}.
Then $
Y_0 < \theta$  holds and hence by Lemma \ref{recurrsion} we obtain $Y_m\to 0$
which gives \eqref{Boundedness_Estimate_Holder}.
\end{proof}

Now we prove so-called ``Density lemma'' (or ``Thin set lemma''):

\begin{lemma}\label{Thin_set}
Assume $u\in DG(\Om, k_0)$. Then there exists
$ \theta_0\in (0,1)$ depending only on the constants $n$, $\ga$, $\al$, $q$ in Definition \ref{Def_DG} of De Giorgi class  such that for any
$ B_{4\rho}:=B_{4\rho}(x_0) \Subset\Om$  if
$$
|B_{2\rho}\cap A_{k_0}| \, \le \,  \theta_0 \, |B_{2\rho}|$$ then either
\begin{equation}\label{Sup_decay}
  \sup\limits_{B_\rho} (u-k_0)_+ \ \le \ \frac 12~
\sup\limits_{B_{4\rho}} (u-k_0)_+
\end{equation}
or
\begin{equation}\label{Sup_decay_1}
 \sup\limits_{B_{4\rho}} (u-k_0)_+  \ \le \ 4c_* \, F \, (2\rho)^{1-\frac nq}
\end{equation}
where $c_*>0$ is a constant from \eqref{Boundedness_Estimate_Holder}.
\end{lemma}

\begin{proof} Take
  $\theta_0\in (0,1)$ such  that
$
c_*~\theta_0^{1/2}   =   \frac 14
$.
Assume \eqref{Sup_decay_1} does not hold.
 Then from \eqref{Boundedness_Estimate_Holder} we obtain
 \eqref{Sup_decay}.
\end{proof}

Now we prove a result called ``How to find a thin set'':

\begin{lemma} \label{How_to_find}
Assume
$u\in DG(\Om, k_0)$. Then for any  $
 \dl\in (0,1)$ and any $ \theta\in (0,1)$ there exists  $s\in \Bbb N$  depending only on $\theta$, $\dl$ and the constants $n$, $\ga$, $\al$, $q$ in Definition \ref{Def_DG} of De Giorgi class such that  if for some $B_{4\rho}:=B_{4\rho}(x_0)\Subset\Om$ the following  estimate  is valid
$$
|B_{2\rho}\setminus A_{ k_0}| \
\ge \ \dl~|B_{2\rho}|
$$
  then either
\begin{equation}\label{How_to_find_est}
|B_{2\rho} \cap
A_{\bar k} | \ \le  \ \theta~ |B_{2\rho}|,
\end{equation}
or
\begin{equation}\label{How_to_find_est_1}
M(4\rho) -k_0 \, \le \,  2^s\, F\, \rho^{1-\frac nq}.
\end{equation}
Here we denote  $
\bar k= M(4\rho)  - \frac 1{2^s}(M(4\rho)  - k_0)$ and     $M(4\rho) \,    := \,
\sup\limits_{B_{4\rho}} u$.
\end{lemma}

\begin{proof} Denote $M:= M(4\rho)$ and assume $s\in \Bbb N$ is some number whose value will be fixed later. From $u\in DG(B_{4\rho})$ for any   $k_0\le k \le \bar k$ we obtain
\begin{equation}\label{M-k}
\int\limits_{ B_{2\rho}\cap A_k} ~ |\nabla u|^2~dx dt
\ \le \  c(n, \ga, \al)~\rho^{n-2} ~\left((M-k)^2 + F^2 \rho^{2(1-\frac nq)}\right).
\end{equation}
Assume \eqref{How_to_find_est_1} is not valid. Then for any $k_0\le k\le \bar k$
$$
F\, \rho^{1-\frac nq} \, \le \, \frac 1{2^{s}}\, (M-k_0) \, = \, M-\bar k \, \le \, M-k
$$
and from \eqref{M-k} we obtain
\begin{equation}\label{M-k-1}
\int\limits_{ B_{2\rho}\cap A_k} ~ |\nabla u|^2~dx dt
\ \le \ 2\,  c(n, \ga, \al)~\rho^{n-2} ~ (M-k)^2  .
\end{equation}
Define
$
k_m:=M-\frac{1}{2^m}(M-k_0)
$  and denote
$$
D_m:= \{ \ x\in B_{2\rho}~|~ k_m \le u(x) < k_{m+1} \ \}.
$$
For simplicity in the rest of the proof of this lemma we will abuse the notation of $A_m$ in Definition \ref{Def_DG} by denoting
$$
A_m:=\{ \ x\in  B_{2\rho} ~|~ u(x)>k_m  \ \}.
$$
From the De Giorgi inequality (see \cite[Lemma 3.5]{LU}) we derive
$$
(k_{m+1}-k_m) |A_{m+1}| \ \le \ C(n) \frac{|\, B_{2\rho} \,
|^{1+\frac 1n}}{|\, B_{2\rho}\setminus A_m \, |} \ \int\limits_{D_m}
|\nabla u(x)|~dx.
$$
Note that
$$
|\, B_{2\rho}\setminus A_m \, | \ \ge \ \dl~ |B_{2\rho}|
$$
and hence
$$
(k_{m+1}-k_m)~ |A_{m+1}| \ \le \ \frac{C(n)}{\dl}~\rho  \
\int\limits_{D_m} |\nabla u|~dx.
$$
Using the H\" older inequality and \eqref{M-k-1} we obtain
$$
(k_{m+1}-k_m) |A_{m+1}| \ \le \ \frac{C(n,\ga, \al) }{\dl} ~\rho^{\frac
n2} \ |\, D_m \, |^{1/2}  (M-k_m).
$$
Note that
$$
k_{m+1}-k_m= \frac{1}{2^{m+1}}(M-k_0), \qquad M-k_m =
\frac{1}{2^{m}}(M-k_0).
$$
Hence
$$
|A_{m+1}|^2 \ \le \ C(n,\ga, \al, \dl) ~\rho^n \ |\, D_m \, |, \qquad
m=1, 2,3,  \ldots
$$
Taking a summation over  $m$ from  1
to $s-1$ we gat
$$
\sum\limits_{m=1}^{s-1} |A_{m+1}|^2 \ \le \ C(n, \ga, \al,  \dl) ~\rho^n
\ \sum\limits_{m=1}^{s-1} |\, D_m \, |.
$$
Note that
$$
\sum\limits_{m=1}^{s-1} |\, D_m \, |  \ = \ |A_1\setminus A_2|
+\ldots+ |A_{s-1} \setminus A_s | \ = \ |A_1\setminus A_s| \ \le \ |
B_{2\rho}|,
$$
$$
A_s\subset A_m, \quad m=1,\ldots, s-1 \qquad\Longrightarrow \qquad
\sum\limits_{m=1}^{s-1} |A_{m+1}|^2 \ \ge \ (s-1) |A_s|^2.
$$
So, for any $s\in \Bbb N$ we obtain
$$
|A_s| \ \le \ \frac{C(n, \ga, \al, \dl)}{\sqrt{s-1}}~\rho^{\frac
n2}~|B_{2\rho}|^{\frac 12}.
$$
Let  $\theta\in (0,1)$ be arbitrary. Find
$s\in \mathbb N$, $s=s(\dl,\theta, n,\ga, \al)$, so that
$$
\frac{C(n, \ga,\al, \dl)}{\sqrt{s-1}}~\rho^{\frac n2}~|B_{2\rho}|^{\frac
12} \ \le \ \theta ~|B_{2\rho}|.
$$
Then
$$
|A_s|\ \le \ \theta ~| B_{2\rho}|,
$$
which implies \eqref{How_to_find_est}.
\end{proof}

Our next result is a local estimate of oscillations:

\begin{lemma}\label{osc_lemma}
Assume   \ $\pm u\in DG(\Om)$. Then there exist constants $\si \in (0,1)$ and $c_1>0$ depending only on the constants $n$, $\ga$, $\al$, $q$ in the Definition \ref{Def_DG} of the De Giorgi class  such  that
\begin{equation}\label{Ost_est}
 \forall\, B_{4\rho}(x_0)  \Subset \Om \qquad  \operatorname{osc}\limits_{B_{\rho}(x_0) } u \ \le \ \si ~\operatorname{osc}\limits_{B_{4\rho}(x_0)} u   \ + \ c_1\,  F \, \rho^{1-\frac nq}.
\end{equation}
\end{lemma}

\begin{proof}
Take any $B_{4\rho}:= B_{4\rho}(x_0)\Subset \Om$ and denote
$$
 m(\rho)=\inf\limits_{B_{\rho}} u, \qquad  M(\rho)=
\sup\limits_{B_{\rho}} u, \qquad k_0:= \frac {m(4\rho)+M(4\rho)}2.
$$
Let   $\theta_0\in (0,1)$ be  the constant from Lemma \ref{Thin_set} and denote by $s\in \Bbb N$ the constant from
Lemma \ref{How_to_find} which corresponds to $\dl =\frac 12$, $\theta= \theta_0$  and the constants $n$, $\ga$, $\al$, $q$ in Definition \ref{Def_DG} of the De Giorgi class. Let $c_*>0$ be the constant from  \eqref{Boundedness_Estimate_Holder} and take
\begin{equation}\label{constant_c_1}
c_1 \ =\ \max\{ 16c_*, 2^{s+1}\}
\end{equation}
and assume
$$
 \operatorname{osc}\limits_{B_{\rho}(x_0) } u \ > \ c_1\,  F \, \rho^{1-\frac nq}.
$$
does not hold.
 Then both \eqref{Sup_decay_1} and \eqref{How_to_find_est_1} are not valid.
Consider two cases:
$$
\gathered
\mbox{(1)} \qquad \big|~\{ \ u \le k_0 \ \}\cap B_{2\rho}~\big|\ \ge
\ \tfrac 12~|B_{2\rho}|,
\\
\mbox{(2)}  \qquad \big|~\{ \ u \ge k_0 \ \}\cap B_{2\rho}~\big|\ \ge
\ \tfrac 12~|B_{2\rho}|.
\endgathered
$$

\medskip
\noindent
Case (1). Take
 $$
 \bar k:=M(4\rho) -\frac 1{2^{s}}(M(4\rho)-k_0).
 $$
Applying Lemma \ref{How_to_find}  with
$\dl=\frac 12$ and  $\theta=\theta_0$ we conclude
$$
|B_{2\rho}\cap A_{\bar k}| \  \le \ \theta_0~|B_{2\rho}|.
$$
Then from Lemma \ref{Thin_set} we obtain
$$
\sup\limits_{B_\rho}(u-\bar k)_+ \ \le \ \frac 12~
\sup\limits_{B_{4\rho}}(u-\bar k)_+,
$$
and hence
$$
M(\rho) \ \le \ \frac 12~\big(M(4\rho)  + \bar k\big),
$$
which means
$$
M(\rho) \ \le \ M(4\rho) - \frac1{2^{s+2}}
\Big(M(4\rho)-m(4\rho)\Big).
$$
Adding this with \ $-m(\rho)\le
-m(4\rho)$ \ we arrive at
$$
\operatorname{osc}\limits_{B_\rho } u \ \le \ \si~\operatorname{osc}\limits_{B_{4\rho}} u ,  \qquad \mbox{where} \qquad \si := 1-\frac
1{2^{s+2}}.
$$

\medskip
\noindent
Case (2). We denote
$$
v:= -u, \qquad l_0:=-k_0, \qquad M_v(\rho) :=\sup\limits_{B_\rho} v,
\qquad m_v(\rho) :=\inf\limits_{B_\rho} v.
$$
Hence  we  have  $v\in DG(\Om)$. From $(2)$ we get
$$
|~\{ v\le l_0\}\cap B_{2\rho}\}~| \ \ge \ \tfrac 12~|B_{2\rho}|.
$$
Hence similar to above
$$
M_v(\rho) \ \le \ M_v(4\rho) - \frac1{2^{s+2}}
\Big(M_v(4\rho)-m_v(4\rho)\Big),
$$
which gives
$$
\operatorname{osc}\limits_{B_\rho } v \ \le \ \left(1-\frac
1{2^{s+2}}\right) ~\operatorname{osc}\limits_{B_{4\rho}} v   ,  \qquad \mbox{where} \qquad \si := 1-\frac
1{2^{s+2}}.
$$
As
$$
\operatorname{osc}\limits_{B_\rho } v \ = \
\operatorname{osc}\limits_{B_\rho } u, \qquad
\operatorname{osc}\limits_{B_{4\rho} } v \ = \
\operatorname{osc}\limits_{B_{4\rho} } u,
$$
the result follows.
 \end{proof}

The  next result follows easily by iterations of the  estimate of oscillations:

\begin{lemma}\label{Holder_continuity}  Assume $u\in L_\infty(\Om)$ satisfies
\eqref{Ost_est} with some $\si\in (0,1)$.
 Then
$$
\forall~B_R(x_0)\subset \Om, \quad \forall\, \rho< R \qquad
 \operatorname{osc}\limits_{B_\rho(x_0)} u
\, \le \, c_2\, \left( \Big(\frac \rho {R}\Big)^\mu  \, \| u\|_{L_\infty(\Om)}  + F\, \rho^\mu\right)
$$
where \ $\mu \in (0,1)$ depends only on $\si$, $n$, $q$ and $c_2$ depends on $\si$, $n$, $q$ and $c_1$.
\end{lemma}

Now part (i) of Theorem \ref{DG_Theorem} follows from Lemma \ref{Holder_continuity}.

\bigskip
Now we turn to the investigation of the boundary regularity.

\begin{lemma}\label{osc_lemma_boundary}
Assume   \ $\pm u\in \pm u\in DG(\Om)\cap DG(\cd \Om)$. Then there exist constants $\si \in (0,1)$, $c_1>0$, $R_1>0$ depending only on the constants $n$, $\ga$, $\al$, $q$, $R_0$, in the Definitions \ref{Def_DG} and \ref{Def_DG_boundary} of the De Giorgi classes and  the Lipschitz constant of  $\cd \Om$  such  that
\begin{equation}\label{Ost_est_boundary}
 \forall\, x_0\in \cd \Om, \quad \forall\, \rho\le R_1 \qquad    \operatorname{osc}\limits_{\Om_{\rho}(x_0) } u \ \le \ \si ~\operatorname{osc}\limits_{\Om_{4\rho}(x_0)} u   \ + \ c_1\,  F \, \rho^{1-\frac nq}.
\end{equation}
Here we denote $\Om_\rho(x_0):=\Om\cap B_\rho(x_0)$.
\end{lemma}

\begin{proof}
Take arbitrary $\Om_0 \Supset \Om$ and denote by $\tilde u$ the zero  extension of $u$ from $\Om$ onto $\Om_0$. Obviously, if $\pm u\in DG(\Om)\cap DG(\cd \Om)$ then $\pm \tilde u \in DG(\Om_0, 0)$ and hence, in particular, $u\in L_\infty(\Om)$.
As $\cd \Om$ is Lipschitz there exists $\dl_0\in (0,1)$ and $0<R_1< \min\left( R_0, \dist\{ \Om, \cd \Om_0\}\right)$ such that for any $x_0\in \cd \Om$ and any $R\in (0,R_1)$
$$
|\{ \, x \in \Om_0: \, \tilde u(x) =0 \,\}\cap B_R(x_0)| \, \ge \, \dl_0 \, |B_R|.
$$
Denote   $\tilde v := -\tilde u$.  Then $\tilde u$, $\tilde v\in DG(\Om_0,0)$. Take any $x_0\in \cd \Om$ and assume $B_{4\rho}(x_0)\Subset \Om_0$. Denote
$$
 m(\rho)=\inf\limits_{B_{\rho}(x_0)} \tilde u, \qquad  M(\rho)=
\sup\limits_{B_{\rho}(x_0)} \tilde u,
$$
$$
 m_{\tilde v}(\rho)=\inf\limits_{B_{\rho}(x_0)} \tilde v, \qquad  M_{\tilde v}(\rho)=
\sup\limits_{B_{\rho}(x_0)} \tilde v,
$$
Note that
$$
m(\rho) \le 0 \le M(\rho)
$$
 and
$$
|\{ \, x \in \Om_0: \, \tilde u(x) \le 0 \,\}\cap B_{2\rho}(x_0)| \, \ge \, \dl_0 \, |B_{2\rho}|,
$$
$$
|\{ \, x \in \Om_0: \, \tilde v(x) \le 0 \,\}\cap B_{2\rho}(x_0)| \, \ge \, \dl_0 \, |B_{2\rho}|.
$$
Let   $\theta_0\in (0,1)$ be  the constant from Lemma \ref{Thin_set} and denote by $s\in \Bbb N$ the constant from
Lemma \ref{How_to_find} which corresponds to $\dl =\frac 12$, $\theta= \theta_0$  and the constants $n$, $\ga$, $\al$, $q$ in the definition of the De Giorgi class $DG(\Om_0, 0)$.
Define $c_1>0$ by  \eqref{constant_c_1}
and assume
$$
 \operatorname{osc}\limits_{B_{\rho}(x_0) } \tilde u \ > \ c_1\,  F \, \rho^{1-\frac nq}.
$$
 Then  \eqref{Sup_decay_1} and \eqref{How_to_find_est_1} are not valid for both $\tilde u$ and $\tilde v$.
Using Lemma \ref{How_to_find}  with
$\dl=\dl_0$,  $\theta=\theta_0$ and $k_0=0$ we obtain
$$
M(\rho) \ \le \ \si~M(4\rho),
\qquad
M_{\tilde v}(\rho) \ \le \ \si ~M_{\tilde v}(4\rho),  \qquad \mbox{where} \qquad \si := 1-\frac
1{2^{s+2}}.
$$
Taking the sum of these two inequalities and taking into account $$M_{\tilde v}(\rho)= -m(\rho),  \qquad M_{\tilde v}(4\rho)= -m(4\rho)$$  we obtain
$$
\operatorname{osc}\limits_{\Om_\rho(x_0) } u \ \le \  \si ~\operatorname{osc}\limits_{\Om_{4\rho}(x_0)} u.
$$
\end{proof}

The next result is the boundary analogue of Lemma \ref{Holder_continuity}.

\begin{lemma}\label{Holder_continuity_boundary}
 Assume $u\in L_\infty(\Om)$ satisfies
\eqref{Ost_est_boundary} with some $\si\in (0,1)$.
 Then
$$
\forall~ x_0 \in \cd\Om, \quad \forall\, \rho< R\le R_1  \qquad
 \operatorname{osc}\limits_{\Om_\rho(x_0)} u
\, \le \, c_2\, \left( \Big(\frac \rho {R}\Big)^\mu  \, \| u\|_{L_\infty(\Om)}  + F\, \rho^\mu\right)
$$
where \ $\mu \in (0,1)$ depends only on $\si$, $n$, $q$ and $c_2$ depends on $\si$, $n$, $q$ and $c_1$.
\end{lemma}

Now part (ii) of Theorem \ref{DG_Theorem} follows by the standard combination of estimates in  Lemmas \ref{Holder_continuity} and \ref{Holder_continuity_boundary}.

\newpage
\section{H\" older continuity of weak solutions} \label{Proof_T3}
\setcounter{equation}{0}

\bigskip
\medskip

In this section we give  the proof of Theorem \ref{Theorem_3}.  Assume $p>2$ and $u\in \overset{\circ}{W}{^1_p}(\Om)$ is a $p$-weak solution to the problem \eqref{Equation} corresponding to the right-hand side $f\in L_q(\Om)$ with $q>n$. From Proposition \ref{Theorem_2} we conclude $u\in L_\infty(\Om)$.

Let us fix some $r\in \left(\frac{2n}{n+2},2\right)$ and  $\theta\in \left(\frac nr -\frac n2,1\right)$. Then from \eqref{Assumptions_Morrey_space} and Proposition \ref{Holder_inequality} we obtain
$$b\in L^{r, n-r}(\Om), \qquad \| b\|_{L^{r, n-r}(\Om)} \, \le \,  c\, \| b\|_{L^{2, n-2}_w(\Om)}.$$
First we consider the case $B_R(x_0)\subset \Om$.
  Take  some radius $\rho<R$ and  a cut-off function $\zeta\in C_0^\infty(B_R(x_0))$ such that  \begin{equation}0\le \zeta\le 1 ,  \qquad \zeta\equiv 1 \quad \mbox{on} \quad B_\rho(x_0), \qquad |\nabla \zeta|\, \le \, \frac{c}{R-\rho}. \label{Cut-off} \end{equation}
   Assume $k\in \Bbb R$ is arbitrary and denote
   $$
   \tilde u \, := \, (u-k)_+ \, \equiv \, \max\{ u-k, 0\}, \qquad \tilde u \in L_\infty(\Om)\cap W^1_p(\Om).
   $$
Fix $m:=\frac 1{1-\theta}$ and note that $2m-1=m(1+\theta)$.
   Take $\eta = \zeta^{2m} \tilde u $  in \eqref{Identity}. It is easy to see that
  $$
  \mathcal B[u, \eta] \ = \  \mathcal B[\zeta^m\tilde u, \zeta^m \tilde u ] \ - \   m\, \int\limits_\Om \zeta^{2m-1} \,  b\cdot \nabla \zeta  \, |\tilde u |^2\, dx.
  $$
  Taking into account   \eqref{Quadratic_Form_Non-negative} and $2m-1=m(1+\theta)$ we obtain
  \begin{equation}\label{From_this}
  \gathered
  \|\zeta^m\nabla\tilde u\|_{L_2(B_R(x_0))}^2 \ \le  \ c\, \| \tilde u \nabla \zeta\|_{L_2(B_R(x_0))}^2  +   m\, \int\limits_\Om \zeta^{m(1+\theta)}\,  b\cdot \nabla \zeta  \, |\tilde u|^2\, dx \ +  \\ + \
  \| f\|_{L_q(\Om)}^2 \, |A_k\cap B_R(x_0) |^{1-\frac 2q}
  \endgathered
  \end{equation}
  where  $A_k:= \{\, x\in \Om: \, u(x)>k\, \}$.   Taking into account  \eqref{Cut-off} and applying the estimate \eqref{Morrey_Compactness_estimate}   we obtain
  $$
  \int\limits_\Om \zeta^{m(1+\theta)}\,  b\cdot \nabla \zeta  \, |\tilde u|^2\, dx  \ \le \  \frac{c\, R^\theta}{ R-\rho  }\, \|b\|_{L^{r, n-r}(\Om)} \, \left\| \nabla (\zeta^m \tilde u )\right\|_{L_2(B_R(x_0))}^{1+\theta }
  \| \tilde u\|_{L_2(B_R(x_0))}^{1-\theta }.
  $$
  Taking arbitrary $\ep>0$ and applying the Young inequality we obtain
  $$
  \gathered
  \int\limits_\Om \zeta^m b\cdot \nabla \zeta  \, |\tilde u |^2\, dx \ \le \ \ep \, \left\| \nabla \left(\zeta^m \tilde u\right)\right\|_{L_2(B_R(x_0))}^2 \ + \\ +  \ \frac{c_\ep}{(R-\rho)^2} \Big( \frac R{R-\rho}\Big)^{\frac {2\theta} {1-\theta}}\,   \|b\|_{L^{r, n-r}(B_R(x_0))}^{\frac { 2} {1-\theta}}     \|   \tilde u \|_{L_2(B_R(x_0))}^2.
  \endgathered
  $$
  So, if we fix sufficiently small $\ep>0$ from \eqref{From_this} for any $k\in \Bbb R$ and $0<\rho<R$ we obtain
  $$
  \gathered
 \tfrac 12\,  \| \nabla (u-k)_+\|_{L_2(B_\rho(x_0))}^2 \ \le \\ \le  \ \frac{c}{(R-\rho)^2}\, \left(1+ \Big( \frac R{R-\rho}\Big)^{\frac {2\theta} {1-\theta}}\,   \|b\|_{L^{r, n-r}(\Om)}^{\frac { 2} {1-\theta}} \right)\,  \| (u-k)_+  \|_{L_2(B_R(x_0))}^2    + \\ + \
  \| f\|_{L_q(\Om)}^2 \, |A_k\cap B_R(x_0) |^{1-\frac 2q}.
  \endgathered
  $$
  Hence we obtain that $u\in DG(\Om)$. Applying the same arguments to $-u$ instead of $u$ we also obtain $-u\in DG(\Om)$.

 Now we turn to the estimates near the boundary. Assume $x_0\in \cd \Om$ and denote $  \Om_R(x_0):=\Om \cap B_R(x_0)$.  Take a cut-off function $\zeta\in C_0^\infty(B_R(x_0))$ satisfying \eqref{Cut-off}. Then for any $k\ge 0$   the function  $\eta = \zeta^{2m} (u-k)_+$  vanishes on $\cd \Om$ and hence it is admissible for the identify \eqref{Identity}. Repeating all previous arguments we arrive at
  $$
  \gathered
 \tfrac 12\,  \| \nabla (u-k)_+\|_{L_2(\Om_\rho(x_0))}^2 \ \le \\ \le   \ \frac{c}{(R-\rho)^2}\, \left(1+ \Big( \frac R{R-\rho}\Big)^{\frac {2\theta} {1-\theta}}\,   \|b\|_{L^{r, n-r}(\Om)}^{\frac { 2} {1-\theta}}\right)\,  \| (u-k)_+  \|_{L_2(\Om_R(x_0))}^2    + \\ + \
  \| f\|_{L_q(\Om)}^2 \, |A_k\cap \Om_R (x_0)|^{1-\frac 2q},
  \endgathered
  $$
  which holds for any $k\ge 0$. Hence $u\in DG(\cd \Om)$.  Applying the same arguments to $-u$ instead of $u$ we also obtain $-u\in DG(\cd \Om)$. Now Theorem \ref{Theorem_3} follows directly from Theorem \ref{DG_Theorem}.

 \newpage
 \section{Appendix}\label{Appendix_2}

\setcounter{equation}{0}

\bigskip
In this section we prove the estimate \eqref{Drift_approximation1}.
\begin{proposition}
  \label{Smooh_approximation} Assume $\Om\subset \Bbb R^n$ is a bounded domain and let $b$ satisfies \eqref{Assumptions_Morrey_space}.
  Assume  $\Om'\Subset \Om$,  $0<\ep< \operatorname{dist}\{ \bar \Om', \cd \Om\}$
  and denote by $b_\ep\in C^\infty(\bar \Om)$ the convolution
  $$
 b_\ep(x)\, := \, \int\limits_{\Om} \om_\ep(x-y)b(y)\, dy, \qquad x\in \Om,
$$
where $\om_\ep(x):=\ep^{-n}\om(x/\ep)$ and $\om\in C_0^\infty(\Bbb R^n)$ is the standard Sobolev kernel satisfying properties \eqref{Sobolev_kernel}.
 Then
  $b_\ep$ satisfies the estimate
  \begin{equation}
\gathered
 \| b_\ep\|_{L^{2,n-2}_w(\Om')} \, \le \, c\, \|b\|_{L^{2,n-2}_w(\Om)}.
 \endgathered
 \label{Drift_approximation}
\end{equation}
with some constant $c>0$ depending only on $n$.

\end{proposition}

\begin{proof}
Take arbitrary $ x_0\in \Om'$. Note that $B_\ep(x)\subset \Om $ for any $ x\in   \Om'$.
First we consider the case $R\in (0, \ep)$.
From the definition of the convolution we have a pointwise estimate
$$
|b_\ep(x)| \, \le \, \frac 1{\ep^n}\, \int\limits_{B_\ep(x)} |b(y)|\, dy, \qquad \forall\, x\in B_R(x_0)\cap  \Om'
$$
and hence from  the definition of the Morrey  norm $L^{1,n-1}(\Om)$ we obtain
$$
|b_\ep(x)| \, \le \, \frac 1{\ep}\, \| b \|_{L^{1,n-1}(\Om)} \,  \qquad \forall\, x\in B_R(x_0)\cap  \Om'.
$$
Taking into account  the estimate
$$
\| b  \|_{ L^{1,n-1}(\Om )}    \, \le \, c  \,  \|b \|_{L^{2,n-2}_w(\Om)}
$$
with a constant $c>0$ depending only on $\Om$  we arrive at
$$
|b _\ep(x)| \, \le \, \frac 1{\ep}\, \| b  \|_{L^{1,n-1}(\Om)} \, \le \,  \frac{c}{\ep}  \,  \|b \|_{L^{2,n-2}_w(\Om)}, \qquad \forall\, x\in B_R(x_0)\cap  \Om'.
$$
Hence we obtain
$$
\| b_\ep\|_{L_{2}(B_R(x_0)\cap  \Om')}^2 \, \le \, c\, \left(\frac R\ep\right)^2 \, \| b\|_{ L^{2,n-2}_w( \Om)}^2\, R^{n-2}.
$$
Applying the trivial estimate $$\| b_\ep\|_{L_{2,w}(B_R(x_0)\cap  \Om')}\le \| b_\ep\|_{L_{2}(B_R(x_0)\cap \Om)}$$
for $R\in (0, \ep)$ we arrive at
\begin{equation}
\label{Approx_estimate}
\| b_\ep\|_{L_{2,w}(B_R(x_0)\cap  \Om')}^2 \, \le \, c\,  \| b\|_{ L^{2,n-2}_w( \Om)}^2\, R^{n-2}.
\end{equation}
Now consider the case $R\in ( \ep , \operatorname{diam}\Om)$.  For  any $f\in L_1(B_{2R}(x_0)\cap \Om)$ we define a convolution operator
$$
(T_\ep f)(x) \, := \, \int\limits_{B_\ep(x)} \om_\ep(x-y) f(y)\, dy, \qquad x\in B_R(x_0)\cap \Om'.
$$
From the H\" older inequality for $p\in [1, +\infty)$ we obtain
$$
|(T_\ep f)(x)|^p \, \le \, \int\limits_{B_\ep(x)} \om_\ep(x-y) |f(y)|^p\, dy, \qquad x\in B_R(x_0)\cap  \Om'.
$$
As for $\ep< \min\left(R, \operatorname{dist}\{ \bar \Om', \cd \Om\}\right)$   we have the inclusion $$\forall\,  x\in B_R(x_0)\cap \Om'  \qquad B_\ep(x)\subset B_{2R}(x_0)\cap  \Om,$$ we conclude
$$
\int\limits_{B_R(x_0)\cap \Om'}  |(T_\ep f)(x)|^p \, dx \, \le \, \int\limits_{B_{2R}(x_0)\cap  \Om} |f(y)|^p\, dy, \qquad x\in B_R(x_0)\cap \Om'.
$$
The last inequality implies that    for $\ep< \min\left(R, \operatorname{dist}\{ \bar \Om', \cd \Om\}\right)$  the linear operator $T_\ep$ is  bounded   from $L_p(B_{2R}(x_0)\cap \Om)$ to $L_p(B_R(x_0)\cap \Om')$ for any $p\in [1, +\infty)$ and the following estimate holds:
$$
M_p\, := \, \| T_\ep\|_{L_p(B_{2R}(x_0)\cap \Om)\to L_p(B_R(x_0)\cap \Om')} \, \le  \, 1.
$$
Using Proposition \ref{Marcinkiewicz} we conclude that the following interpolation estimate holds:
$$
\| T_\ep\|_{L_{2,w}(B_{2R}(x_0)\cap \Om)\to L_{2,w}(B_R(x_0)\cap \Om')} \, \le \,  c\, M_3^{1/2} \, M_{3/2}^{1/2}  \, \le \, c
$$
with some constant $c>0$ depending only on $n$. Hence we obtain
$$
\| b_\ep\|_{L_{2,w}(B_R(x_0)\cap \Om')} \, \le \, c\, \| b \|_{L_{2,w}(B_{2R}(x_0)\cap \Om)}.
$$
From the definition of the Morrey quasinorm $L^{2,n-2}_w(\Om)$ we obtain
$$
\| b_\ep\|_{L_{2,w}(B_R(x_0)\cap \Om')}^2 \, \le \, c\, \| b \|_{L_w^{2,n-2}(\Om)}^2 \, (2R)^{n-2}.
$$
Hence  for any $R\in ( \ep , \operatorname{diam}\Om)$ we arrive at the estimate \eqref{Approx_estimate}. So, the estimate \eqref{Approx_estimate} is valid for all $R\in (0,\operatorname{diam}\Om)$. Hence  for any $\Om'\Subset \Om$ and any $\ep\in (0, \operatorname{dist}\{ \bar \Om', \cd \Om\})$ we obtain \eqref{Drift_approximation}.
\end{proof}

\end{document}